\documentclass[a4paper, 10pt, twoside, notitlepage]{amsart}

\usepackage[utf8]{inputenc}
\usepackage{color}
\usepackage{amsmath} 
\usepackage{amssymb} 
\usepackage{amsthm}
\usepackage{geometry}
\usepackage{graphicx}
\usepackage{esint}
\usepackage[colorlinks=true,linkcolor=blue]{hyperref}
\usepackage{mathtools}
\definecolor{BudGreen}{RGB}{112, 174, 110}

\theoremstyle{plain}
\newtheorem{thm}{Theorem}
\newtheorem{prop}{Proposition}[section]
\newtheorem{lem}[prop]{Lemma}

\newtheorem{rmk}[prop]{Remark}

\newcommand {\R} {\mathbb{R}} 
 \newcommand {\N} {\mathbb{N}}
 
\newcommand {\p} {\partial}

\DeclareMathOperator*{\esup}{ess\,sup}

\newcommand{\K}{\mathcal{K}}
\newcommand{\E}{E_{\epsilon}}
\newcommand{\F}{F_{\epsilon}}

\newcommand{\tv}{\tilde v}
\newcommand{\nabF}{\nabla F_{|F^{-1}(ry)}}
\newcommand{\osc}{{\rm osc}}

\DeclareMathOperator {\dist} {dist}

\DeclareMathOperator{\inte} {int}

\title[On the Effect of Geometry]{On the Effect of Geometry on Scaling Laws for a Class of Martensitic Phase Transformations}
\author{Janusz Ginster}
\address{Institut f\"ur Mathematik, Humboldt-Universit\"at zu Berlin, Unter den Linden 6, 10099 Berlin, Germany}
\email{janusz.ginster@hu-berlin.de}
\author{Angkana Rüland}
\address{Institute for Applied Mathematics and Hausdorff Center for Mathematics, University of Bonn, Endenicher Allee 60, 53115 Bonn, Germany}
\email{rueland@uni-bonn.de}
\author{Antonio Tribuzio}
\address{Institute for Applied Mathematics, University of Bonn, Endenicher Allee 60, 53115 Bonn, Germany}
\email{tribuzio@iam.uni-bonn.de}
\author{Barbara Zwicknagl}
\address{Institut f\"ur Mathematik, Humboldt-Universit\"at zu Berlin, Unter den Linden 6, 10099 Berlin, Germany}
\email{barbara.zwicknagl@hu-berlin.de}

\begin{document}

\begin{abstract}
We study scaling laws for singular perturbation problems associated with a class of two-dimensional martensitic phase transformations and deduce a domain dependence of the scaling law in the singular perturbation parameter. In these settings the respective scaling laws give rise to a selection principle for specific, highly symmetric domain geometries for the associated nucleation microstructure. More precisely, firstly, we prove a general lower bound estimate illustrating that in settings in which the domain and well geometry are \emph{incompatible} in the sense of the Hadamard-jump condition, then necessarily at least logarithmic losses in the singular perturbation parameter occur in the associated scaling laws. Secondly, for specific phase transformations in two-dimensional settings we prove that this gives rise to a dichotomy involving logarithmic losses in the scaling law for \emph{generic} domains and optimal linear scaling laws for \emph{very specific, highly compatible} polygonal domains. In these situations the scaling law thus gives important insight into optimal isoperimetric domains. We discuss both the geometrically linearized and nonlinear settings.
\end{abstract}

\maketitle

\section{Introduction}

It is the objective of this article to study the interaction of the domain geometry and the symmetry of the set of stress-free states for certain vector-valued, multi-well singular perturbation problems associated with martensitic phase transformations. More precisely, we deduce a \emph{quantitative effect of the domain geometry} on scaling laws associated with these singular perturbation problems in settings which allow for special, very low energy microstructures when considered with austenite boundary conditions. Prototypical examples include the geometrically linear and nonlinear hexagonal-to-rhombic phase transformations as well as the square-to-rectangular phase transformation. For this class of transformations, we prove the following dichotomy: On the one hand, for a \emph{very specific class of domains} which are in their symmetries ``optimally adapted'' to the phase transformation, one has an optimal \emph{linear} scaling behaviour in the singular perturbation parameter. On the other hand, however, \emph{``generic'' domains} necessitate additional boundary microstructure, giving rise to \emph{logarithmic losses} in the small perturbation parameter. In both settings, we deduce upper and lower bound estimates.
To the best of our knowledge this is one of the first results precisely quantifying the geometry dependence in non-degenerate domains by relating the domain geometry and the geometry of the martensitic energy wells. It provides a rigorous justification of the emergence of very special nucleation microstructure with domain symmetries inherited from the well symmetries in these phase transformations. It can also be viewed as a quantitative version of the Hadamard jump condition. In these situations, the scaling laws thus provide detailed information on optimal domains in the associated isoperimetric problems.

\subsection{The setting}

Let us begin by outlining the model we are studying. We consider variational, singular perturbation problems associated with martensitic phase transformations both in the geometrically nonlinear and the linearized settings. 

\subsubsection{The geometrically nonlinear model}
\label{sec:nonl_mod}

Within the geometrically nonlinear context, we investigate scaling laws for (sharp-interface) energies consisting of an elastic and a surface contribution
\begin{align}
\label{eq:nonlinear_main}
E_{\epsilon}(u,\Omega) := \int_{\Omega} \dist^2(\nabla u,\mathcal{K}) \, dx + \epsilon |D^2 u|(\Omega).
\end{align}
Here $\Omega \subset \R^2$ denotes the reference configuration, $u: \Omega \rightarrow \R^2$ models the deformation, $| \cdot|(\Omega)$ denotes the total variation norm, and the set $\K \subset \R^{2\times 2}$ corresponds to the stress-free states of the martensitic phase transformation, also called the energy wells. In the applications to martensitic transformations, the energy wells $\K$ typically are of the form 
\begin{align*}
\K := \bigcup\limits_{j=1}^{m} SO(2) U_j
\end{align*}
 with $U_j \in \R^{2\times 2}_{sym,>}$ being symmetry related (see the applications in Sections \ref{sec:appl_lin} and \ref{sec:appl_nonlin} below). From a materials science point of view, the wells $SO(2) U_j$ represent the different phases of the materials, i.e., the so-called austenite phase and/or the so-called variants of martensite. The $SO(2)$ invariance of the energy wells reflects the frame indifference of the model. We refer to \cite{BJ87,BJ92,B,M1,B04} for further background on the modelling of these materials.

\subsubsection{The geometrically linearized model}
\label{sec:lin_mod}

In parallel to our discussion of the full, geometrically nonlinear model \eqref{eq:nonlinear_main}, in this article, we also simultaneously discuss associated geometrically linearized models. While preserving the \emph{material nonlinearity}, i.e., the multi-well structure of the model \eqref{eq:nonlinear_main}, in the geometrically linearized setting, the \emph{geometric nonlinearity}, i.e., the frame invariance assumption, is linearized. This leads to $Skew(2)$ instead of $SO(2)$ invariance. As a consequence, in these settings, we study energies of the form
\begin{align}
\label{eq:linear_main}
\F(v,\Omega) := \int_{\Omega} \dist^2(e( v),\tilde{\mathcal{K}}) \, dx + \epsilon |D^2 v|(\Omega).
\end{align}
Here, as above, $\Omega \subset \R^2$ is the reference configuration, $v: \Omega \rightarrow \R^2$ denotes the (infinitesimal) displacement with the infinitesimal strain tensor $e(v):= \frac{1}{2}(\nabla v + (\nabla v)^t)$, and the set $\tilde{\mathcal{K}}$ models the stress-free states of the martensitic phase transformation and will also be referred to as the energy wells. In the geometrically linearized setting, these are typically of the form
\begin{align*}
\tilde{\mathcal{K}} = \{e^{(1)},\dots, e^{(m)}\},
\end{align*}
with $e^{(j)} \in \R^{2\times 2}_{sym}:= \{M \in \R^{2\times 2}: \ M = M^{t}\}$ for $j\in \{1,\dots,m\}$. We refer to \cite[Chapter 11]{B} and \cite{Bhat93} for comparisons of the geometrically linearized and nonlinear settings. Both in definition \eqref{eq:nonlinear_main} and \eqref{eq:linear_main} we occasionally also drop the domain dependence in cases that there is no danger of confusion on which domain the energies are considered.

\medskip

In both the settings of \eqref{eq:nonlinear_main}, \eqref{eq:linear_main}, it is our main objective, to investigate the interaction of the domain geometry and the geometry and symmetries of the set of wells $\K$ and $\tilde{\K}$, respectively. To this end, we deduce scaling laws for the minimal energy contribution in terms of the singular perturbation parameter for a certain class of wells $\K$ and $\tilde{\K}$ and quantify the domain dependence in terms of matching upper and lower scaling bounds. As one of our main results, we infer that the interaction of the domain geometry and the geometry of the wells leads to the \emph{selection of special nucleation microstructure} (c.f. Theorems \ref{thm:main_gen} and \ref{thm:main}).

\subsection{The lower bound estimates}
\label{sec:intro_lower}

We begin by discussing general \emph{lower} scaling bounds, both for the geometrically nonlinear and linear models \eqref{eq:nonlinear_main}, \eqref{eq:linear_main}. In the next sections, we will complement these results with matching \emph{upper} bound estimates for suitable choices of the sets $\K$ and $\tilde{\K}$.

As previously explained, the logaritmic loss in the lower scaling bounds below arises from an incompatibility between the wells $\K$ and the geometry of a part of the boundary of the domain $\Omega$ on which a boundary condition is imposed.
To state our result in the most general setting, we consider boundary conditions of oscillation type.
For this reason, we introduce the following notion of \emph{oscillation of traces} of a function $v\in W^{1,2}(\Omega;\R^2)$ on $\tilde \Gamma\subset \p\Omega$;
$$
\osc_{\tilde\Gamma}(v):=\esup_{x_1,x_2\in\tilde\Gamma}|v_{|\p\Omega}(x_1)-v_{|\p\Omega}(x_2)|.
$$

For the geometrically nonlinear model, our main lower bound estimate, which will be applied to specific sets $\K$ in the next subsections, reads:

\begin{thm}
\label{prop:gen_domain_nonlinear}
Let $\mathcal{K} \subseteq \R^{2\times 2}$ be compact and let $\Omega \subset \R^2$  be a compact Lipschitz domain. Let $\Gamma \subset \partial \Omega$ be relatively open and $C^3$ regular. Let $\tau: \Gamma \rightarrow \R^2$ denote  a $C^2$ regular unit tangent vector field to $\Gamma$.
 If  
\begin{align}
\label{eq:dist}
 d:=\min\limits_{x\in\bar\Gamma} \min\limits_{K \in \mathcal{K}} | K \tau(x) - \tau(x)| > 0,
 \end{align}
  then 
  for any $u \in W^{1,2}(\Omega;\R^2)$ such that   
    \begin{equation}\label{eq:bc-nonlin-gen}
  	\begin{split}
  		\osc_{\tilde\Gamma}(u-id)\le\frac{d}{208}\mathcal{H}^1(\tilde\Gamma), \text{ for every relatively open connected } \tilde\Gamma\subset\Gamma,
  	\end{split}
  \end{equation}
  
  it holds
\[
\E(u,\Omega) = \int_{\Omega} \dist^2(\nabla u,\mathcal{K}) \, dx + \epsilon |D^2u|(\Omega) \geq c \min\{ 1, \epsilon (|\log \epsilon| + 1) \}.
\] 
\end{thm}

Let us comment on this: Theorem \ref{prop:gen_domain_nonlinear} provides a ``generic'' lower bound scaling law for settings in which there is no special compatibility relating the domain $\Omega$, the boundary conditions on $\Gamma$ and the  structure of the wells $\K$. In a precise way this is encoded in the condition \eqref{eq:dist} together with the assumption \eqref{eq:bc-nonlin-gen}. Compared to the ``trivial'' bound which scales linearly in $\epsilon$, the logarithmic contributions quantify the mismatch between the geometry of the domain and the wells, enforcing a (weak) presence of microstructure. More precisely, we highlight that the condition \eqref{eq:dist} can be interpreted as a failure of the Hadamard-jump condition, with Theorem \ref{prop:gen_domain_nonlinear} encoding this quantitatively. 
This is most transparent, if $u|_{\overline{\Gamma}} = id|_{\overline{\Gamma}}$, i.e., if austenite boundary conditions are prescribed.
Indeed, in this case by the Hadamard jump condition, a flat austenite-martensite interface with normal $\tau^{\perp} \in S^1$ between the identity deformation and a variant of martensite given by one of the deformation matrices $K \in \K$ would correspond to the condition that
\begin{align}
\label{eq:Hadamard_nonlinear}
(K-Id) \tau = (a \otimes \tau^{\perp}) \tau = 0 \mbox{ for some } a \in \R^2.
\end{align}
Here, for a vector $\nu \in \R^2$ we write $\nu^{\perp}$ for its clockwise rotation by $90^{\circ}$ and define $a\otimes \tau^{\perp}:= a (\tau^{\perp})^t$ which is understood in the sense of matrix multiplication.
For any fixed boundary point $x\in \Gamma$, the assumption \eqref{eq:dist} thus rules out that for any $K \in \K$ the property \eqref{eq:Hadamard_nonlinear} holds. In this sense, it enforces incompatibility between the wells and the boundary data and thus leads to the presence of microstructure. In case that the boundary conditions are not perfect austenite conditions, the assumption \eqref{eq:bc-nonlin-gen} should then be read as a quantification of how much the boundary conditions deviate from the austenite. The interaction of the smallness condition \eqref{eq:bc-nonlin-gen} and the well compatibility condition \eqref{eq:dist} then imply that the Hadamard jump condition between the boundary data and the energy wells is still violated. We remark that for technical reasons, for general domains with non-flat boundaries, we rely on the condition \eqref{eq:bc-nonlin-gen} even in the case of austenite data (see the proof in Section \ref{sec:proof_lower}). If stronger regularity assumptions on $u$ were available, this would not be necessary. From a technical point of view, our results are related to and inspired by the works \cite{GZ23, GZ23a} on low energy microstructures. In the present work, we extend these ideas to settings involving gauge invariances ($SO(2)$ and $\text{Skew}(2)$), fully vectorial models as well as to rather general geometries.

We further observe that, in the typical case of nonlinear elasticity in which $\K=\bigcup_{j=1}^m SO(2)U_j$ with $U_j\in\R^{2\times2}_{sym,>}$, condition \eqref{eq:dist} can be reduced to the following bound
$$
d:=\min_{x\in\bar{\Gamma}}\min_{j\in\{1,\dots,m\}}\big||U_j\tau(x)|-1\big|>0.
$$
Indeed, due to rotation invariance, in this case, a failure of \eqref{eq:Hadamard_nonlinear} is possible only if $U_j$ are non-identical deformations in direction $\tau$. Vice versa, as encoded in the Hadamard jump condition, compatible microstructure without the logarithmic losses necessitates tangential continuity. The condition \eqref{eq:dist} hence quantifies the violation of the tangential continuity property.

We note that for bounded $C^3$ regular domains, by the hairy ball theorem, the condition \eqref{eq:dist} is always satisfied on a (sufficiently small) portion of $\partial \Omega$ if $\K$ is a discrete set. Such $C^3$ austenite-martensite interfaces are thus closely connected to \emph{non-classical austentite-martensite interfaces} as studied in \cite{BC97, BC99, BKS09, BJK14}.
Similarly, in the class of piecewise polygonal domains, i.e., for domains with boundaries which consist of a union of finitely many planes, if $\K$ is finite, generically the condition \eqref{eq:dist} holds.
Thus, it is only for very special choices of piecewise polygonal domains and wells, that it is possible to improve on these ``generic'' $\epsilon |\log(\epsilon)|$ bounds and obtain optimal linear scaling laws. We remark that for physical reasons (as the surface energies are believed to often carry only small prefactors), the settings for small choices of the parameter $\epsilon>0$ are of particular physical interest.

In the subsequent sections, we will complement the general lower bound of Theorem \ref{prop:gen_domain_nonlinear} with the observation that in the very non-generic setting that the domain $\Omega$, the boundary conditions on $\Gamma$ and the  structure of the wells $\K$ are in a precise way related to each other, it is possible to improve on these bounds and to prove an optimal linear scaling law. We interpret this as a selection mechanism for nucleation microstructures.

The situation in the geometrically linearized setting is fully parallel to this. Also in this setting, unless there is a very special relation between the domain geometry, the boundary conditions and the structure of the wells, one obtains a lower scaling law with logarithmic losses, which in general cannot be avoided.

\begin{thm}
\label{prop:gen_domain_lin}
Let $\tilde{\mathcal{K}} \subseteq \R^{2\times 2}_{sym}$ be compact and let $\Omega \subset \R^2$ be a bounded Lipschitz domain. Let $\Gamma \subset \partial \Omega$ be relatively open and $C^3$ regular.  Let $\tau: \Gamma \rightarrow \R^2$ denote a $C^2$ regular unit tangent vector field to $\Gamma$.
 If  
\begin{align}
\label{eq:dist-lin}
 d:=\min\limits_{x\in\bar\Gamma} \min\limits_{K \in \tilde{\mathcal{K}}} | \tau(x)\cdot K \tau(x) | > 0,
 \end{align}
  then  
  for any $v \in W^{1,2}(\Omega;\R^2)$ such that
    \begin{equation}\label{eq:bc-lin-gen}
  	\begin{split}
  		\osc_{\tilde\Gamma}(v) \le \frac{d}{64\cdot17\cdot9\cdot\sqrt{2}}\mathcal{H}^1(\tilde\Gamma), \text{ for every relatively open connected } \tilde\Gamma\subset\Gamma,
  	\end{split}
  \end{equation}
    
  it holds
\[
\F(v,\Omega) = \int_{\Omega} \dist^2(e( v), \tilde{\mathcal{K}}) \, dx + \epsilon |D^2 v|(\Omega) \geq c \min\{ 1, \epsilon (|\log \epsilon| + 1) \}.
\] 
\end{thm}

We emphasize that the assumptions \eqref{eq:dist-lin}, \eqref{eq:bc-lin-gen} are direct geometrically linear analogues of \eqref{eq:dist}, \eqref{eq:bc-nonlin-gen}.
Indeed, by the Hadamard jump condition the existence of a flat austenite-martensite interface between the zero displacement strain and a strain $K \in \tilde{\K}$ with normal direction $\tau^{\perp} \in S^1$ is equivalent to the condition that 
\begin{align*}
\tau \cdot K \tau = (\tau \cdot a\odot \tau^{\perp}) \tau = 0 , \mbox{ for some } a \in \R^2.
\end{align*} 
Here we have used the notation $a\odot \tau^{\perp}:= \frac{1}{2}(a\otimes \tau^{\perp} + \tau^{\perp} \otimes a)$.
For $x\in  \Gamma$ fixed, condition \eqref{eq:dist-lin} encodes the absence of such a flat austenite-martensite interface between the zero strain and any of the matrices in the set $\tilde{\K}$ with any of the normals $\tau(x)^{\perp}$. As above, this hence enforces the presence of microstructure and eventually leads to the logarithmic losses analogously as in the nonlinear setting of Theorem \ref{prop:gen_domain_nonlinear}.

\subsection{A geometry- and symmetry-induced selection of microstructure: The geometrically linear hexagonal-to-rhombic phase transformation}
\label{sec:appl_lin}
In this subsection, it is our objective to show how Theorem \ref{prop:gen_domain_lin} gives rise to a dichotomy in the scaling law for certain phase transformations allowing for highly symmetric microstructure. To this end, we will focus on a special case of the geometrically linearized two-dimensional $n$-well differential inclusions investigated in \cite{CDPRZZ20} (and also in \cite{RZZ16} and references therein). We focus on the case of three martensitic wells as this corresponds to the experimentally interesting case of the hexagonal-to-rhombic phase transformation.
In this case (for a fixed temperature below the transformation temperature) we set
\begin{align}
\label{eq:Kn_lin}
\tilde{\K}_3:=\{e^{(1)}, e^{(2)}, e^{(3)} \},
\end{align}
with
\begin{align*}
e^{(1)} =  \frac{1}{2}\begin{pmatrix}
1 & -\sqrt{3} \\ -\sqrt{3} & -1 \end{pmatrix}, \ 
\ e^{(2)} = 
\begin{pmatrix} -1 & 0 \\ 0 & 1 \end{pmatrix},\
e^{(3)} = \frac{1}{2}\begin{pmatrix} 1 & \sqrt{3} \\ \sqrt{3} & -1\end{pmatrix}.
\end{align*}
We note that the matrices $e^{(j)}$ are obtained from the matrix $e^{(1)}$ by the action of an element in the point group of the hexagonal lattice, i.e., $e^{(j)} = Q_j e^{(1)} Q_j^t$ for $j\in \{1,2,3\}$ and 
\begin{align}
\label{eq:rot}
Q_j := Q(\frac{2\pi}{3}(j-1)), \ 
Q(\varphi) := \begin{pmatrix} \cos(\varphi) & - \sin(\varphi) \\ \sin(\varphi) & \cos(\varphi) \end{pmatrix}.
\end{align}

For this phase transformation, the exactly stress-free states at temperatures below the transformation temperature are described by the differential inclusion
\begin{align*}
e(v) \in \tilde{\mathcal{K}}_3 \mbox{ a.e. in } \Omega,
\end{align*}
where $v: \Omega \rightarrow \R^2$ denotes the displacement.

We recall the following facts, which give evidence of the observation that the microstructures forming in the hexagonal-to-rhombic phase transformation are quite flexible and can be concatenated to obtain rather complex structures:
\begin{itemize}
\item All pairs $e^{(i)}, e^{(j)}$ with $i \neq j$ are pairwise symmetrized rank-one connected, i.e., there exist $n_{ij} \in S^1, \ a_{ij} \in \R^2 \setminus \{0\}$ such that
\begin{align*}
e^{(i)}- e^{(j)} = a_{ij} \odot n_{ij}.
\end{align*}
In particular, it is possible to form simple laminate solutions.
\item There exist (more complicated) exactly stress-free microstructures involving corners of higher order. More precisely there are:
\begin{itemize}
\item a single configuration involving a twelve-fold corner,
\item (up to symmetry) two types of six-fold corners,
\item (up to symmetry) one type of four-fold corner.
\end{itemize}
We refer to Figure 24 in \cite{RZZ16} for an illustration of these building block microstructures.
\item For austenite boundary conditions there are ``star-type'' microstructures in very specific domains which are given as certain rotated equilateral triangles (see Figure 25, 26 in \cite{RZZ16} as well as the (geometrically linearized) constructions in \cite{CDPRZZ20}). We emphasize that nonlinear analogues of these can be found in \cite{CKZ17} and \cite{CDPRZZ20} and have been observed experimentally \cite{KK91}. A schematic illustration can also be found in Figure \ref{fig:star}.
\end{itemize}

We remark that the hexagonal-to-rhombic phase transformation is also of relevance experimentally as a two-dimensional reduction of the three-dimensional transformation occurring in materials such as Mg$_2$Al$_4$Si$_5$O$_{18}$, Mg-Cd or Pb$_3$(VO$_4$)$_2$ alloys \cite{KK91,CPL14,MA80,MA80a}.

\begin{figure}
\includegraphics[width = 3 cm]{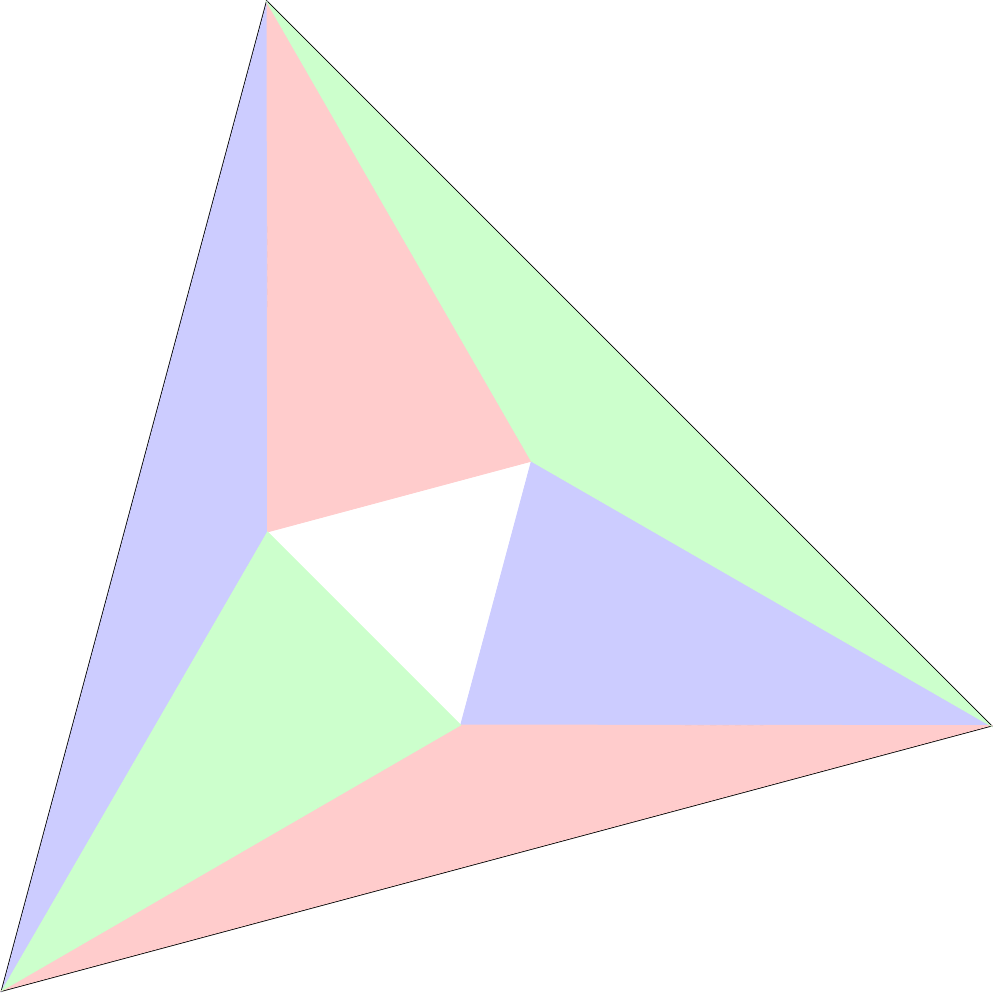}
\hspace{2cm}
\includegraphics[width = 3cm]{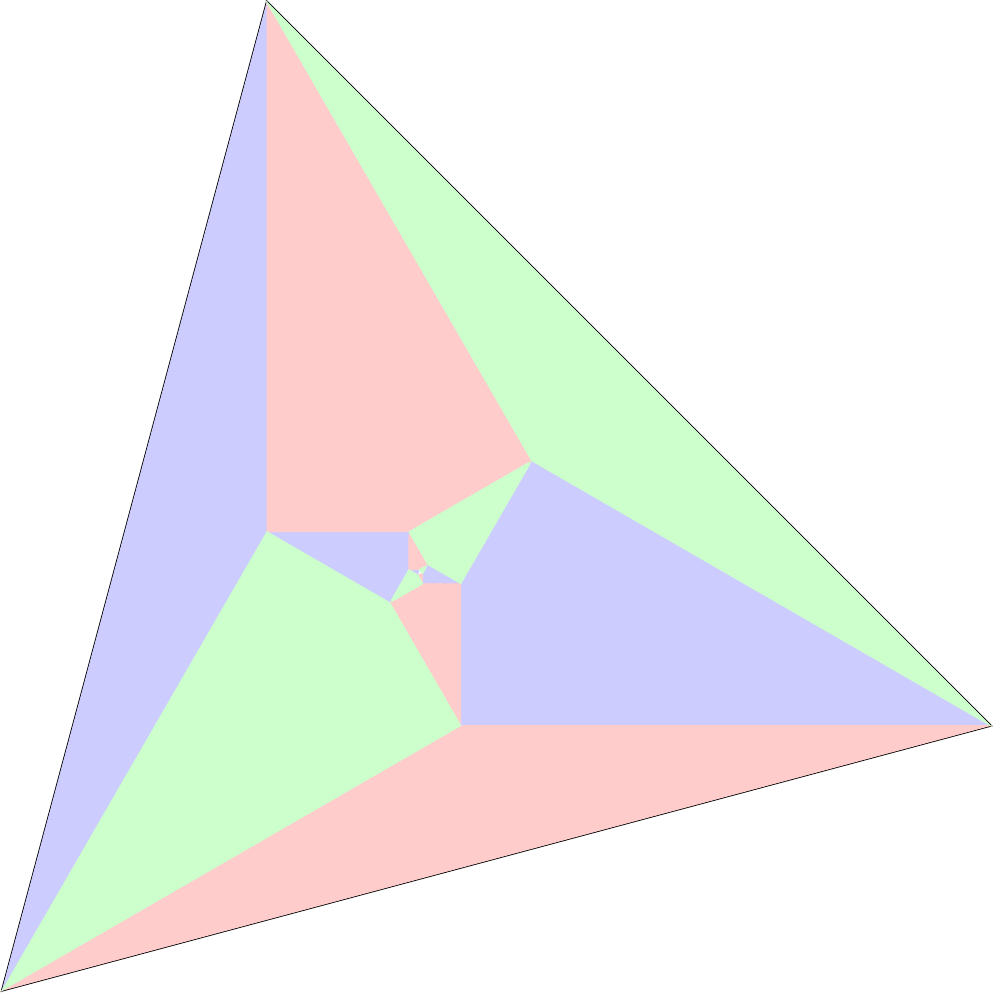}
\caption{The self-similar, star-type low energy microstructure in the hexagonal-to-rhombic phase transformation. The figure corresponds to the reference domain, the colours to the variants of martensite $e^{(1)}, e^{(2)}, e^{(3)}$ by which the corresponding domains are displaced. On the left panel, only the outer ring of the displacement is depicted. Here both on the outer and inner boundary condition austenite data are prescribed. The right panel shows the geometrically refining, self-similar construction of the overall microstructure with austenite prescribed on the outer boundary of the triangle and the interior now being fully covered by a martensitic microstructure.}
\label{fig:star}
\end{figure}

We note that all the matrices in $\tilde{\K}_3$ are symmetrized rank-one connected to the austenite and that these normals can be explicitly computed:

\begin{lem}[Austenite-martensite normals]
\label{lem:normals}
Let $\tilde{\K}$ be as in \eqref{eq:Kn_lin} and let $Q_j \in SO(2)$ be as in \eqref{eq:rot}. Then all normals associated with flat austenite-martensite interfaces are given (up to normalization and orientation reversal) by the set
\begin{align}
\label{eq:normals}
\mathcal{\tilde{N}}:=\{ Q_j (e_1+e_2), Q_j (e_1-e_2); j \in \{1,2,3\}\}.
\end{align}
\end{lem}

As a consequence, it is possible to form two types of laminates between each variant of martensite and austenite using one of these normals.

In the outlined setting of the set $\tilde{\K}_3$ from \eqref{eq:Kn_lin} for general $n\in \mathcal{\tilde{N}}$, we deduce a scaling law, which takes into account the geometry of the underlying domain in the following dichotomy.

\begin{thm}
\label{thm:main_gen}
Let $\Omega \subset \R^2$ be an open, bounded Lipschitz domain and let $\epsilon>0$.  Let $\tilde{\mathcal{K}} = \tilde{\K}_3$ and let $\F(v; \Omega)$ be as in \eqref{eq:linear_main}.
Then the following holds:
\begin{itemize}
\item[(a)] Let $\Omega$ be a bounded, piecewise $C^3$-regular domain such that there exists a relatively open subset $\Gamma\subset\p\Omega$ for which $\nu(x)\not\in\mathcal{\tilde{N}}$ for every regular point $x\in\Gamma$.
Then it holds
\begin{align*}
 \inf\limits_{v \in W^{1,2}(\Omega;\R^2) \text{ with } v_{|\Gamma} = 0} \F(v;\Omega) \sim_{\Omega,\Gamma}  \min \{ 1, \epsilon (|\log \epsilon| + 1) \}.
\end{align*} 
\item[(b)] For any normal $n\in \mathcal{\tilde{N}}$ there exists a bounded polygonal domain with one interface having $n$ as a normal and such that 
\begin{align*}
 \inf\limits_{v \in W^{1,2}(\Omega;\R^2) \text{ with } v_{|\partial \Omega} = 0}  \F(v;\Omega) \sim_{\Omega} \min\{\epsilon,1\}.
\end{align*}
\end{itemize}
\end{thm}

Let us comment on the dichotomy formulated in the theorem.
As shown in \cite{RZZ16} and in \cite{CDPRZZ20} for the setting from \eqref{eq:Kn_lin}, for a set of highly symmetric, polygonal domains $\Omega$ it is possible to obtain $BV$ regular solutions to the differential inclusion $e(v)\in \tilde{\K}_3 \text{ a.e. in } \Omega$. As a consequence, this implies that for these specific polygonal domains, a scaling law which is linear in the singular perturbation parameter is valid. This is the content of Theorem \ref{thm:main_gen}(b). 
As one of our main results, we prove that this linear scaling behaviour is very non-generic. More precisely, Theorem \ref{thm:main_gen}(a) asserts that for arbitrary piecewise $C^3$ domains (with tangent to the boundary in some regular point that does not match a flat interface with one of the normals $n \in \tilde{\mathcal{N}}$ allowing for an austenite-martensite twin, one immediately obtains a scaling law which deviates from the linear scaling by a logarithmic factor. Here the lower bound follows from the general result of Theorem \ref{prop:gen_domain_nonlinear} while the upper bound follows from a construction which refines dyadically towards the boundary, similarly as in \cite[Section 7c]{B3}, see also Figure \ref{fig:covering} below. The resulting austenite-martensite interfaces are thus \emph{non-classical} in (a geometrically linearized variant of) the sense of \cite{BC97, BC99}.
Hence, our result provides a rigorous foundation for the expectation that for the models under consideration nucleation of martensite in austenite is accompanied with the formation of nuclei of a very particular geometry which is closely related to the geometry of the martensitic wells.

\subsection{A geometry- and symmetry-induced selection of microstructure: The geometrically nonlinear hexagonal-to-oblique and square-to-oblique transformations}
\label{sec:appl_nonlin}

We next consider the geometrically nonlinear setting and prove a similar result for certain one-parameter families of geometrically nonlinear two-dimensional deformations, building on the structures from \cite{CKZ17} and \cite{CDPRZZ20}. For the ease of the presentation, we focus here on two particular examples of one-parameter families of transformations, the square-to-oblique transformation and the hexagonal-to-oblique transformation studied in \cite{CKZ17}, although the presented results hold for a larger class of transformations, see Remark \ref{rmk:gen_nonlin}. 

Let us describe the associated families of deformations. These are determined by the symmetries of the transformation. For the square-to-oblique case a square is mapped to a parallelogram in the passage from the austenite to the martensite phase.  
Due to the symmetries of the point group, we are lead to study the following set of stress-free matrices 
\begin{align}
\label{eq:K4}
\K_4:= \left\{Q_4(a) U_1,\dots, Q_4(a) U_4 \right\}, \ \mbox{ with } U_1 = \begin{pmatrix} a & a^{-1}-a \\ 0 & a^{-1} \end{pmatrix},
\end{align}
with $a >0$ and with $Q_4(a) \in SO(2)$ being an explicit rotation matrix which depends on $a$ (see for instance \cite[Remark 2.3]{CDPRZZ20} where $Q_4(a) = R^t$).
The matrices $U_j:= P_j U_1 P_j^t$ are here obtained by conjugation of $U_1$ with the matrices 
\begin{align*}
P_j \in \mathcal{P}_c:=\left\{ Id, \begin{pmatrix}  -1 & 0 \\ 0 & 1 \end{pmatrix}, \begin{pmatrix} 0 & -1 \\ 1 & 0 \end{pmatrix}, \begin{pmatrix} 0 & 1 \\ 1 & 0 \end{pmatrix} \right\}
\end{align*}
forming the point group of the square austenite lattice.

Similarly, when the austenite lattice is of hexagonal structure and transforms into a parallelogram, the associated stress-free states take the form

\begin{align}
\label{eq:K3}
\K_3:= \left\{Q_3(a)U_1,\dots,Q_3(a)U_6 \right\}, \ \mbox{ with } U_1 = \begin{pmatrix} a & \sqrt{3}(a^{-1}-a) \\ 0 & a^{-1} \end{pmatrix},
\end{align}
with $a>0$ and where $Q_3(a)$ is an explicit rotation matrix which depends on $a$ (see, for instance, \cite[Remark 2.3]{CDPRZZ20} where $Q_3(a) = R^t$).
The matrices $U_j:= P_j U_1 P_j^t$ are here obtained by conjugation of $U_1$ with the matrices 
\begin{align*}
P_j \in \mathcal{P}_h:=\left\{ Id, \begin{pmatrix}  -1 & 0 \\ 0 & 1 \end{pmatrix}, \frac{1}{2} \begin{pmatrix} 1 & \sqrt{3} \\ - \sqrt{3} & 1 \end{pmatrix}, \frac{1}{2} \begin{pmatrix} -1 & \sqrt{3} \\ \sqrt{3} & 1 \end{pmatrix},    \frac{1}{2} \begin{pmatrix} 1 & -\sqrt{3} \\ \sqrt{3} & 1 \end{pmatrix}, \frac{1}{2} \begin{pmatrix} 1 & \sqrt{3} \\ \sqrt{3} & -1 \end{pmatrix} \right\}
\end{align*}
forming the point group of the hexagonal austenite lattice.

By virtue of the results from \cite{CKZ17} and \cite{CDPRZZ20} there is a finite set of at most eight or twelve normal vectors, respectively, for which elements in $\K_3$ or $\K_4$ are compatible with the austenite phase. We denote these directions by the sets $ \mathcal{N}_3$ or $\mathcal{N}_4$, respectively. We remark that these sets are of the form
\begin{align*}
\mathcal{N}_j := \{P n_k P^t: P \in \mathcal{P}_j, \ k\in \{1,2\} \}, \ j \in \{3,4\}
\end{align*}
for some vectors $n_1, n_2 \in S^1$ and for $\mathcal{P}_3 = \mathcal{P}_h $ and $\mathcal{P}_4 = \mathcal{P}_c $. The normals $n_1, n_2$ can be computed explicitly (depending on $a$). As these are rather convoluted, we do not give the explicit expressions here.

Inserting these sets $\K_3, \K_4$ into the energy \eqref{eq:nonlinear_main}, we then obtain the following dichotomy which is analogous to the one of Theorem \ref{thm:main_gen}:

\begin{thm}
\label{thm:main}
Let $\Omega \subset \R^2$ be an open, bounded Lipschitz domain and let $\epsilon>0$.  Let $\E(u; \Omega)$ be as in \eqref{eq:nonlinear_main} and let $\K_3,\K_4$ be as in \eqref{eq:K4}, \eqref{eq:K3}. Consider the choices $\mathcal{K} = \mathcal{K}_3$ or $\K = \K_4$ in the geometrically nonlinear energy \eqref{eq:nonlinear_main}. 
Then the following holds:
\begin{itemize}
\item[(a)] Let $\Omega$ be a bounded, piecewise $C^3$-regular domain such that there exists a relatively open subset $\Gamma\subset\p\Omega$ for which $\nu(x)\not\in\mathcal{N}_3$ (or $\mathcal{N}_4$, respectively) for every regular point $x\in\Gamma$.
Then,
\begin{align*}
 \inf\limits_{u \in W^{1,2}(\Omega;\R^2) \text{ with } u_{|\Gamma} = id} \E(u;\Omega) \sim_{\Omega,\Gamma}  \min \{ 1, \epsilon (|\log \epsilon| + 1) \}.
\end{align*} 
\item[(b)] For any normal $n\in \mathcal{N}_{3}$ or $n\in \mathcal{N}_{4}$ there exists a bounded, polygonal domain such that $\partial \Omega = \partial \Omega_E \cup \partial \Omega_I$ with at least one interface of $\partial \Omega_E$ having $n$ as a normal and such that 
\begin{align*}
 \inf\limits_{u \in W^{1,2}(\Omega;\R^2) \text{ with } u_{|\partial \Omega_E} = id}  \E(u;\Omega) \sim_{\Omega} \min\{\epsilon,1\},
\end{align*}
where we consider the energies with wells $\K_3$ or $\K_4$, respectively.
\end{itemize}
\end{thm}

In parallel to Theorem \ref{thm:main_gen}, Theorem \ref{thm:main} provides a dichotomy in which the extremely low energy scaling law strongly depends on the interaction of the symmetries of the underlying domain, the imposed boundary data and the energy wells. On the one hand, by virtue of Theorem \ref{thm:main}(b) in very special, highly non-generic domains, the optimal linear scaling result can be achieved. This however requires extreme compatibility between all involved symmetries. On the other hand, in general domains, as soon as a part of the boundary of the domain does no longer match the symmetries of the wells, the energy scaling deteriorates by a logarithmic factor. As in the geometrically linear setting, the scaling law thus \emph{selects nucleation geometries} based on the symmetries of the energy wells.

\begin{rmk}
\label{rmk:gen_nonlin}
The results which were formulated above for the two cases of the square-to-oblique and the hexagonal-to-oblique transformations remain valid, when replacing these sets $\K_3, \K_4$ by a more general set $\K_n$ for $n\in \N$.
In this setting, the matrix $U_1 = U_1(a,n)$ takes the form
\begin{align*}
U_1(a,n) = \begin{pmatrix}
a & \frac{a^{-1}-a}{\tan(\phi)} \\
0 & a^{-1}
\end{pmatrix},
\end{align*}
where $a\in \R_+$ and $\phi = \frac{n-2}{2n} \pi$. As in the special cases outlined above, the remaining stress-free deformations are then obtained by conjugation with a symmetry group of a regular $n$-gon (generalizing the action of the point group). 
For a more precise discussion on the structure of the sets $\K_n$ which generalize the sets $\K_3, \K_4$ to an arbitrary polygonal symmetry we refer to \cite[Section 2]{CDPRZZ20}. The existence of the desired low energy constructions is the content of \cite[Theorem 1]{CDPRZZ20}. In order to avoid introducing the associated notation, we have opted not to formulate the full analogue of the scaling law of Theorem \ref{thm:main} in the most generality at this point.
\end{rmk}

\subsection{Interpretation, implications and comparison with the literature}
Let us put the results from the previous sections into the context of the literature on martensitic phase transformations.

\subsubsection{Relation to the literature}
Seeking to understand the complex microstructure formed by martensitic materials more quantitatively and starting with the seminal results of Kohn-Müller \cite{KM1, KM2}, singular perturbation problems and associated models have been analysed intensively. A non-exhaustive list includes the works \cite{C1, L06, ContiSchweizer06, JerrardLorent2013, DM1, CT05,CO, CO1, ContiChan14a, CC15, BG15, CZ16, CDZ17, CDMZ20, TS21, TS21b, DF20, Rue16b, AKKR23, RT23b, TZ23} focusing on two-well and/or on branching microstructures, the articles \cite{RT22, RT23a,RT23} on particularly high energy structures, \cite{C99, KMS03, MS, CS13, RZZ16, RZZ18,RTZ19} for non-uniqueness results and the works \cite{KK,Pompe,BK16,Z14,KKO13, KO19, CKZ17, CDMZ20, CDPRZZ20, RT23, GZ23, GZ23a,G23} on nucleation and/or low energy structures. Many of these results are also closely connected to complex multi-scale microstructures arising for other non-convex, multi-well problems in materials \cite{K07}, see for instance \cite{CKO99, KW14, KW16, KM13,RRT23, RRTT24, GZ23}. The results in the present article build partly on techniques developed there and connect to the study of low energy microstructures for nucleation, focusing on the role of geometry. Viewing our models as depending on the domain as a parameter, Theorems \ref{thm:main_gen} and \ref{thm:main}, as some of our main results, prove that scaling laws can give detailed information on optimal nucleation structures for specific classes of two-dimensional martensitic phase transformations. 

\subsubsection{Interpretation and implications}
Viewing the examples from the previous sections as two-dimensional model cases for the study of low energy nucleation results, the outlined explicit dependence of the scaling laws on the domain geometry, rigorously explains the formation of special, highly symmetric nucleation structures for these model problems. In contrast to other settings, in these models the energy scaling laws from Theorems \ref{thm:main_gen} and \ref{thm:main} thus do not only predict the \emph{size of the  energy barriers} for nucleation of martensite from austenite. In addition to this already quantitative information which however usually only indirectly carries information on the nucleation microstructure, our scaling laws also \emph{yield precise structure conditions on the domain geometry of minimally scaling configurations}. More precisely, they can be viewed as \emph{selecting} microstructures in which the domain geometry and the symmetries of the energy wells are highly matching. As a consequence, our results predict nucleation geometries of a certain polygonal structure, depending on the energy well symmetry in the model settings of the transformations in Theorems \ref{thm:main_gen} and \ref{thm:main}.

\subsection{Outline of the article}
The remainder of the article is structured as follows: In Section \ref{sec:proof_lower} we deduce the general lower bound estimates from Theorems \ref{prop:gen_domain_nonlinear} and \ref{prop:gen_domain_lin}. Then, in Section \ref{sec:theorems_proof}, we apply these to infer the proofs of Theorems \ref{thm:main_gen} and \ref{thm:main}.

\section{Proof of the lower scaling bound in generic domains}
\label{sec:proof_lower}

We begin our discussion of the effect of geometry on our class of two-dimensional martensitic phase transformations by deducing the general lower bound estimates contained in Theorems \ref{prop:gen_domain_nonlinear} and \ref{prop:gen_domain_lin}. This is split into two parts. In Section \ref{sec:square_model} we first deduce these in the special case of a square domain with ``austenite-type'' boundary conditions. In Section \ref{sec:trafo} we transform these to general piecewise $C^3$ domains.

\subsection{Proof of the lower scaling bound in the square domain}
\label{sec:square_model}

In this section, we present a first variant of the lower bound estimates from Theorems \ref{prop:gen_domain_nonlinear} and \ref{prop:gen_domain_lin} for the special case of $\Omega:=[0,1]^2$ being the square domain. This is inspired by the arguments from \cite{GZ23,GZ23a} which we extend to the fully vectorial setting in the presence of gauges.
In the next section, using suitable coordinate transformations, we will then transfer these results to more general domains.

\begin{prop}
\label{prop:ref_domain}
Let $\mathcal{K} \subseteq \R^{2\times 2}$ and $\tilde{\K} \subset \R^{2\times 2}_{sym}$ be compact. 
\begin{enumerate}
    \item \underline{\emph{Geometrically nonlinear setting:}} assume that  
\begin{align}
\label{eq:incomp_norm1}    
     d:=\min_{K \in \mathcal{K}} |K e_2 - e_2| > 0,
\end{align}     
      then there exists a constant $c>0$ such that for all $u \in W^{1,2}((0,1)^2;\R^2)$ with  
\begin{align}
	\label{eq:trace_nonlin}  
	\begin{split}  
		\osc_{\{0\}\times I}(u-id)\le\frac{d}{104}|I|, \text{ for every open interval $I\subset(0,1)$},
	\end{split}
\end{align}
and all $\epsilon>0$ the following lower bound holds:
\[
\E(u) := \int_{(0,1)^2} \dist^2(\nabla u,\mathcal{K}) \, dx + \epsilon |D^2u|((0,1)^2) \geq c \min\{ 1, \epsilon (|\log \epsilon| + 1) \}.
\] 
    \item \underline{\emph{Geometrically linear setting:}} assume that 
\begin{align}
\label{eq:incomp_norm2}    
    d:=\min_{\tilde{K} \in \tilde{\mathcal{K}}} |\tilde{K}_{22}| > 0,
\end{align}    
     then there exists a constant $c>0$ such that for all $u \in W^{1,2}((0,1)^2;\R^2)$ with   
\begin{align}
	\label{eq:trace_lin}    
	\begin{split}
		\osc_{\{0\}\times I}(u)\le\frac{d}{104}|I|, \text{ for every open interval $I\subset(0,1)$},
	\end{split}
\end{align} 
and all $\epsilon>0$ the following lower bound holds:
\[
\F(u) := \int_{(0,1)^2} \dist^2(e( u),\tilde{\mathcal{K}}) \, dx + \epsilon |D^2u|((0,1)^2) \geq c \min\{ 1, \epsilon (|\log \epsilon| + 1) \}.
\] 
\end{enumerate}
In both cases the constant $c>0$ only depends on the lower bound $d>0$ and on $\K$ (resp.\ $\tilde\K$). 
\end{prop}

Let us comment on the assumptions of the proposition. As in Theorem \ref{prop:gen_domain_nonlinear} the condition \eqref{eq:incomp_norm1} corresponds to the failure of the Hadamard jump condition between martensite and austenite at any point of the boundary portion given by $x=0$. It hence implies that at this part of the boundary necessarily microstructure emerges in the energy minimization problem with austenite boundary conditions. This entails the presence of the logarithmic losses in the lower bound estimate, provided one assumes austenite boundary conditions. The second assumption \eqref{eq:trace_nonlin} allows to generalize the boundary data from austenite data to small perturbations of austenite in a way that preserves the failure of the Hadamard jump condition and thus the incompatibility of the wells at the boundary. It quantifies the smallness which is necessary to still obtain logarithmic losses in the lower bound estimate. In our proof of Theorem \ref{prop:gen_domain_nonlinear} this will allow us to work with curvilinear boundaries. The conditions in the geometrically linearized setting are analogues of this.

Approaching the proof of Proposition \ref{prop:ref_domain}, we begin by discussing an auxiliary result. In this context, it will be convenient to introduce localized versions of the energies from \eqref{eq:nonlinear_main} and \eqref{eq:linear_main}. To this end, we set
\begin{align}
\label{eq:en_localized}
\begin{split}
E_{\epsilon}(u,\omega)&:= \int\limits_{\omega} \dist^2(\nabla u, \mathcal{K}) dx + \epsilon |D^2 u|(\omega),\\
F_{\epsilon}(u,\omega)&:=  \int\limits_{\omega} \dist^2(e( u), \tilde{\mathcal{K}}) dx + \epsilon |D^2 u|(\omega),
\end{split}
\end{align}
for any $\omega \subset \Omega$ such that the quantities in the definitions of \eqref{eq:en_localized} are well-defined. 
In addition, for $x \in (0,1)$ and $I\subseteq (0,1)$ Lebesgue-measurable, we write for $u \in W^{1,2}(\Omega;\R^2)$ with $\nabla u \in BV(\Omega;\R^{2\times 2})$
\begin{align}
\label{eq:slice_energy}
E_{\epsilon}(u;\{x\} \times I) := \int_{I} \dist^2(\nabla u(x,y), \mathcal{K}) \, dy + |\partial_2 \nabla u(x,\cdot)|(I).
\end{align}
Note that since $\nabla u \in BV((0,1)^2;\R^{2\times2})$ this formula makes sense for almost every $x \in (0,1)$ in the sense of slicing of $BV$-functions, see \cite{AFP2000}. 
We use an analogous notation for $F_{\epsilon}(u,\{x\} \times I)$.

\begin{lem}\label{lem: lb log}
Let $\mathcal{B} \subseteq W^{1,2}((0,1)^2)$.
Assume that there exist constants $ \alpha,\beta \in ( 0,1]$ such that for almost all $x \in (\frac{\alpha}2 \min\{ 1, \epsilon (|\log \epsilon| + 1) \},\alpha)$ and all $u \in \mathcal{B}$ it holds
\[
\E(u;\{x\} \times (0,1)) \geq \beta \min\{ \frac{\epsilon}x, 1 \}.
\]
Then there exists a constant $c > 0$ such that
\[
\inf_{u \in \mathcal{B}} \E(u;\Omega) \geq c \min\{\epsilon (|\log \epsilon|+1),1\}.
\]
The same holds if one replaces $\E$ by $\F$.
\end{lem}

\begin{proof}
We only show the statement for $\E$. By assumption it holds 
\[
\inf_{u \in \mathcal{B}} \E(u; \Omega) \geq \int_{\frac{\alpha}2 \min\{ 1, \epsilon (|\log \epsilon| + 1) \}}^{\alpha} \beta \min\{ \frac{\epsilon}x, 1 \} \, dx.
\]
Assume first that $\epsilon \leq \epsilon_0$ for some $0 < \epsilon_0 \leq 1$. 
We notice that $\min\{\frac{\epsilon}x,1 \} \geq \epsilon \frac{\alpha}{2x}$ for all $x \geq \frac{\alpha}2 \epsilon$.
Hence, it follows that
\begin{align}
\begin{split}
\int_{\frac{\alpha}2 \min\{ 1, \epsilon (|\log \epsilon| + 1)\}}^{\alpha} \beta \min\{ \frac{\epsilon}x, 1 \} \, dx &\geq \frac{\alpha \beta}{2} \epsilon  \left( |\log  \epsilon|  + \log (2) - \log (|\log \epsilon| + 1) \right) \\ & \geq \frac{\alpha \beta}{4} \epsilon (|\log \epsilon| + \log(2))
\end{split}
\end{align}
for $\epsilon_0$ small enough. This implies the assertion for $c \leq \frac{\alpha \beta \log(2)}{4 }$.
On the other hand, we have for $\epsilon \geq \epsilon_0$ 
\[
\int_{\frac{\alpha}2 \min\{ 1, \epsilon (|\log \epsilon| + 1)\}}^{\alpha} \beta \min\{ \frac{\epsilon}x, 1 \} \, dx \geq \int_{\alpha/2}^{\alpha} \beta \min\{ \frac{\epsilon_0}{\alpha}, 1\} \, dx = \frac{\alpha\beta}2 \min\{ \frac{\epsilon_0}{\alpha}, 1\},
\]
which implies the assertion for $c\leq \frac{\alpha\beta}2 \min\{ \frac{\epsilon_0}{\alpha}, 1\}$. Hence, we define $c = \frac{\alpha\beta \log(2)}4 \min\{ \frac{\epsilon_0}{\alpha},1 \}$ which implies the desired claim.
\end{proof}

With Lemma \ref{lem: lb log} in hand, we turn towards the proof of Proposition \ref{prop:ref_domain}.

\begin{proof}
We start with the geometrically nonlinear setting and define the following constants $c_2 := \min\{\frac{d}{13}, \frac{1}{13} \}$,
$c_3 := \max_{K \in \mathcal{K}} |K|$ and $c_1 := \frac12 + 16 \frac{c_3}{c_2} \geq 1/2$. Without loss of generality, we assume that $\E(u) \leq \frac{c_2^2}{4c_1 \cdot 8192} \min\{1, \epsilon (|\log \epsilon| + 1)\}$.
Under these assumptions, we will show the following claim. \\

\underline{\textit{Claim:}} There exists a constant $c>0$ such that for almost all $x \in (\frac{1}{4c_1} \min\{ 1, \epsilon (|\log \epsilon| + 1) \},\frac1{2c_1})$,  all $u\in W^{1,2}(\Omega)$ satisfying \eqref{eq:trace_nonlin}  and all $\epsilon > 0$ it holds that
\[
\E(u;\{x\} \times (0,1)) \geq c \min\{ \frac{\epsilon}x, 1 \}.
\]
Given the claim, the assertion of the geometrically nonlinear statement of Proposition \ref{prop:ref_domain} follows from Lemma \ref{lem: lb log} above. It thus remains to provide the argument for the claim.

To this end, let $x \in (\frac{1}{4c_1} \min\{ 1, \epsilon (|\log \epsilon| + 1) \},\frac1{2c_1})$. 
Then, let $I \subseteq (0,1)$ be such that $|I| = c_1 x$,
\begin{equation}
\label{eq:localize}
\E(u; \{x\} \times I) \leq 8 |I| \E(u;\{x\} \times (0,1)), \text{ and } \E(u; (0,1) \times I) \leq 8 |I| \E(u).
\end{equation}
By definition of the energy (see \eqref{eq:slice_energy}), as consisting of an elastic and a surface contribution, on $I$ one of the following estimates holds: 
\begin{enumerate}
    \item[(a)] $|\partial_2 \nabla u(x,\cdot)|(I) \geq c_2$,
    \item[(b)] $\dist(\nabla u(x,y), \mathcal{K}) \geq c_2$ for a.e. $y \in I$,
    \item[(c)] $\dist(\nabla u(x,y), \mathcal{K}) \leq 3c_2$ for a.e. $y \in I$.
\end{enumerate}
Indeed, either the deformation gradient always stays close to the wells (as in case (c)), or is always far away from the wells (as in case (b)) or jumps between being close and far from the wells at some points in $I$ (as in case (a)).
If (a) or (b) hold, then 
\[
\E(u;\{x\} \times I) \geq c_2 \min\{ \epsilon, c_2 |I|\},
\]
which, by \eqref{eq:localize}, implies that
\[ 
\E(u; \{x\} \times (0,1) ) \geq \frac{c_2}{8 |I|} \min\{ \epsilon, c_2 |I|  \} = \frac{c_2}{8} \min\{ \frac{\epsilon}{c_1 x}, c_2\}
\]
and hence results in the claim.
Therefore, it suffices to show the claim assuming that (c) holds but (a) and (b) do not hold.
Then it follows from our assumption \eqref{eq:incomp_norm1} on the quantitative failure of the Hadamard jump condition, and (c) for almost every $y \in I$
\[
|\partial_2 u(x,y) - e_2| \geq d - \dist(\nabla u(x,y), \mathcal{K}) \geq 10 c_2.
\]
This implies that for almost every $y \in I$ 
\[
|\partial_2 u_1(x,y) | \geq 5c_2 \text{ or } |\partial_2 u_2(x,y) - 1| \geq 5c_2.
\]
As we assume that (a) does not hold, we obtain
\[
|\partial_2 u_1(x,y) | \geq 4c_2 \text{ for a.e. } y \in I  \text{ or } |\partial_2 u_2(x,y) - 1| \geq 4c_2 \text{ for a.e. } y\in I.
\]
Again, as (a) does not hold, this implies that one of the following holds:
\begin{enumerate}
    \item $\partial_2 u_1(x,y) \geq 4c_2$ for a.e. $y\in I$,
        \item $\partial_2 u_1(x,y) \leq -4c_2$ for a.e. $y\in I$,
            \item $\partial_2 u_2(x,y) - 1\geq 4c_2$ for a.e. $y\in I$,
                \item $\partial_2 u_2(x,y) - 1 \leq - 4c_2$ for a.e. $y\in I$.
\end{enumerate}
Without loss of generality, we assume that (1) holds.
We define $v: I \to \R$ such that $v'(y) = K(y)_{12}$ where $K: I \to \mathcal{K}$ is a measurable function such that $\dist(\nabla u(x,y),\mathcal{K}) = |\nabla u(x,y) - K(y)|$.
It follows by (1) and (c) that $v'(y) \geq c_2$ for a.e. $y\in I$.
This yields
$$
\frac{c_2}4 |I|^2 \leq \min_{a \in \R} \| v - a\|_{L^1(I)}.
$$

Next, note that  for any  $y_0\in I$ for $a_0:=u_1(0,y_0)$ it holds 
\begin{align*}
	\| u_1(x,\cdot)-a_0\|_{L^1(I)} &\le \| u_1(0,\cdot)- u_1(0,y_0)\|_{L^1(I)} + \| u_1(x,\cdot)- u_1(0,\cdot)\|_{L^1(I)} \\
	&\le \int_I \esup_{y_1,y_2\in I}|u_1(0,y_1)-u_1(0,y_2)|dy + \| u_1(x,\cdot)- u_1(0,\cdot)\|_{L^1(I)} \\
	&\leq \frac{c_2}{8}|I|^2 + \|\p_1 u_1\|_{L^1((0,x)\times I)} \leq \frac{c_2}{8}|I|^2 + x c_3|I| + \|K_{11}-\p_1 u_1\|_{L^1((0,x)\times I)} \\
	& \leq \frac{c_2}{8}|I|^2 + c_3 x |I| + x^{1/2} |I|^{1/2} \E(u; (0,1) \times I)^{1/2}.
\end{align*}
Hence, for $a_1:=\fint_I (u_1(x,y)-v(y))dy$ it follows by the triangle inequality, Poincar\'e's inequality and H\"older's inequality
\begin{align} \label{eq: lb NS triangle}
\begin{split}
    \frac{c_2}4 |I|^2 &\leq 
     \|v- a_1 - a_0\|_{L^1(I)} \\
    & \leq \| u_1(x,\cdot) - v - a_1\|_{L^1(I)} + \| u_1(x,\cdot)- a_0  \|_{L^1(I)}\\
    &\leq |I| \| \partial_2 u_1(x,\cdot) - v' \|_{L^1(I)} +\frac{c_2}{8}|I|^2 + c_3 x |I| + x^{1/2} |I|^{1/2} \E(u; (0,1) \times I)^{1/2} \\
    &\leq |I|^{3/2} \| \dist(\nabla u(x, \cdot), \mathcal{K}) \|_{L^2(I)}+\frac{c_2}{8}|I|^2  + c_3 x |I| + \sqrt{8} x^{1/2} |I| \E(u)^{1/2} \\ 
    &\leq \sqrt{8} |I|^2 \E(u;\{x\} \times (0,1))^{1/2}+\frac{c_2}{8}|I|^2  +  c_3 x |I| + \sqrt{8} x^{1/2} |I| \E(u)^{1/2}.
    \end{split}
\end{align}
This means that it holds 
\begin{align*}
\frac{c_2}{8} |I|^2 \leq 2 \sqrt{8} |I|^2 \E(u;\{x\} \times (0,1))^{1/2} \mbox{ or } \frac{c_2}{8} |I|^2 \leq 2 c_3 x |I| + 2 \sqrt{8} x^{1/2} |I| \E(u)^{1/2}.
\end{align*}
In the first case, we find $\E(u;\{x\} \times (0,1)) \geq \frac{c_2^2}{2048}$ which implies the claim in this case.
In the second case, we estimate 
\[
\frac{c_2}{16} |I|x \leq \frac{c_2}{8} |I|^2 - 2 c_3 x |I| \leq 2 \sqrt{8} x^{1/2} |I| \E(u)^{1/2},
\]
and, consequently,
\[
x \leq \frac{8192}{c_2^2} \E(u) \leq \frac1{4c_1} \min\{1, \epsilon (|\log \epsilon| + 1)\},
\]
i.e., $x$ is irrelevant for the claim.
This finishes the proof of the claim in the geometrically nonlinear setting.
\medskip

Let us now turn to the geometrically linearized setting. 
To this end, we define the following constants $c_5 := \frac{d}7$, $c_6 := 2 \max_{\tilde{K} \in \tilde{\mathcal{K}}} |\tilde{K}|$, $c_4 := 2 + 8\frac{c_6}{c_5} \geq 2$ and $c_7:=\frac{c_5 c_4^2}{8(c_4+1)^2}$.
Moreover, we assume that $\F(u) \leq \frac{c_5^2}{256^2c_4} \min\{ \epsilon (|\log\epsilon| + 1),1\}$. 
As before, the assertion follows from Lemma \ref{lem: lb log} once we show the following claim. \\

\underline{\textit{Claim:}} There exists a constant $c>0$ such that for almost all $x \in (\frac{1}{4c_4} \min\{ 1, \epsilon (|\log \epsilon| + 1) \},\frac1{2c_4})$, all $u\in W^{1,2}(\Omega;\R^2)$ satisfying \eqref{eq:trace_lin}  and all $\epsilon > 0$ it holds
\[
\F(u;\{x\} \times (0,1)) \geq c \min\{ \frac{\epsilon}x, 1 \}.
\]

Similarly to the definition above let $x \in (\frac{1}{4c_4} \min\{ 1, \epsilon (|\log \epsilon| + 1) \},1/(2c_4)) \subseteq (0,1/4)$, let $I \subseteq (1/4,3/4)$ be such that $|I| = c_4 x$ and
\[
\F(u;\{x\} \times I) \leq 16 |I| \, \F(u; \{x\} \times (0,1)) \text{ and } \F(u; S_i) \leq 16 |I| \, \F(u), \ i \in \{1,2\},
\]
where $S_1 = \{ (x-t,y-t) : y \in I,0< t < x \}$ and $S_2 = \{ (x-t,y+t) : y \in I,0< t < x \}$. 
Then, as in the geometrically nonlinear setting, one of the following conditions holds on $I$:
\begin{enumerate}
    \item[(d)] $|\partial_2 \nabla u(x,\cdot)|(I) \geq c_5$,
    \item[(e)] $\operatorname{dist}(e( u) (x,y), \K) \geq c_5$ for almost every $y\in I$,
    \item[(f)] $\operatorname{dist}(e( u) (x,y), \K) \leq 3 c_5$ for almost every $y\in I$.
\end{enumerate}
Again, if (d) or (e) holds, then
\[
16|I| \F(u;\{x\} \times (0,1)) \geq \F(u;\{x\} \times I) \geq c_5 \min\{ \epsilon, c_5 |I|\},
\]
which implies the claim.
Hence, we will assume again that (f) holds but (d) and (e) do not hold.
Then it follows by the assumption \eqref{eq:incomp_norm2} that for a.e.~$y \in I$
\[
|\partial_2 u_2(x,y)| \geq d - \operatorname{dist}(e( u)(x,y),\tilde{\K}) \geq 4 c_5.
\]
Since (d) does not hold, it follows that
\[
\partial_2 u_2(x,y) \geq 4 c_5 \text{ for a.e. $y \in I$ or } \partial_2 u_2(x,y) \leq -4 c_5 \text{ for a.e. $y \in I$.}
\]
Without loss of generality, we assume that $\partial_2 u_2(x,y) \geq 4 c_5$ for a.e.~$y \in I$.
Then define $w: I \to \R$ such that $w'(y) =  \tilde{K}_{22}(y)$, where $\tilde{K}: I \to \tilde{\K}$ is a measurable function such that $\operatorname{dist}(e(u)(x,y),\tilde{\K})=|e(u)(x,y)- \tilde{K}(y)|$. 
It follows by (f) (and the assumption that $\partial_2 u_2(x,y) \geq 4 c_5$) that $w'(y) \geq c_5$ for almost every $y \in I$.

Next, let $\xi = \begin{pmatrix} 1 \\ 1 \end{pmatrix}$ and let $y_0\in I$. 
Using the smallness of the boundary conditions \eqref{eq:trace_lin}, we estimate
\begin{align}
\begin{split}
\label{eq: est uxi}
    &\int_{I} |(u(x,y)-u(0,y_0))\cdot \xi| \, dy \\
    & \qquad = \int_{I} |u(x,y) \cdot \xi - u(0,y-x) \cdot \xi| \, dy +\int_{I} |(u(0,y-x) - u(0,y_0)) \cdot \xi| \, dy \\
    & \qquad \leq \int_{S_1} |\partial_{\xi} (u (s,t)\cdot \xi)| \, ds dt + c_7(|I|+x)^2\\
    & \qquad \leq \int_{S_1} \min_{\tilde{K} \in \tilde{\K}} |e(u)(s,t)\xi \cdot \xi - \tilde{K} \xi \cdot \xi| + c_6 \, ds dt  +c_7(|I|+x)^2 \\
    & \qquad \leq |S_1|^{1/2} \left( \int_{S_1} \operatorname{dist}^2(e(u)(s,t),\tilde{\mathcal{K}}) \, ds dt \right)^{1/2} + c_6 |S_1|  +c_7(|I|+x)^2 \\
    & \qquad \leq 4 x^{1/2} |I| \F(u)^{1/2} + c_6 x |I|  +c_7(|I|+x)^2.
 \end{split}
\end{align}
An analogous argument for $\eta = \begin{pmatrix} 1 \\ -1 \end{pmatrix}$ and $S_2$
yields
\begin{align}\label{eq: est ueta}
    \int_I |(u(x,y)-u(0,y_0)) \cdot \eta| \, dy \leq 4 x^{1/2} |I| \F(u)^{1/2} + c_6 x |I|  +c_7(|I|+x)^2.
\end{align}

Hence, combining \eqref{eq: est uxi} and \eqref{eq: est ueta} by using that $u_2 = \frac{1}{2}u \cdot (\xi-\eta)$, we find
\[
\int_I |(u_2(x,y)-u_2(0,y_0)| \, dy \leq 4 x^{1/2} |I| \F(u)^{1/2} + c_6 x |I|  +c_7(|I|+x)^2.
\]
Then, noticing that $c_7(|I|+x)^2\le\frac{c_5}{8}|I|^2$, arguing similarly as in \eqref{eq: lb NS triangle}, we obtain

\begin{align*}
\frac{c_5}{4} |I|^2 
&\leq \|w- \fint\limits_{I} w(y)+ u_2(x,y) dy - u_2(0,y_0) \|_{L^1(I)}\\
& \leq \|u_2(x, \cdot) - w- \fint\limits_{I} w(y)+ u_2(x,y) dy\|_{L^1(I)}
+ \|u_2(x, \cdot) - u_2(0,y_0)\|_{L^1(I)}\\
& \leq |I| \| \p_2 u_2 - w'\|_{L^1(I)} + 4 x^{\frac{1}{2}} |I|\F(u)^{1/2} + c_6 x |I| + \frac{c_5}{8}|I|^2\\
& \leq 4 |I|^2 \F(u;\{x\} \times (0,1))^{1/2} +  4 x^{\frac{1}{2}} |I|\F(u)^{1/2} + c_6 x |I| + \frac{c_5}{8}|I|^2
\end{align*}

As a consequence,
\begin{align}
\label{eq:ref-last-ineq}
\frac{c_5}{8}|I|^2 \leq 
4 |I|^2 \F(u;\{x\} \times (0,1))^{1/2} + c_6 x |I| + 4  x^{1/2} |I| \F(u)^{1/2}.
\end{align}
Finally, the claim follows by the same case distinction as in the nonlinear setting.
\end{proof}

\subsection{Extension to general domains}
\label{sec:trafo}

\label{sec:gen_domains}

In this section, we extend the previous results to the setting of domains with curved boundaries. Here it suffices to argue locally and to locally transform part of the boundary $\Gamma \subset \partial \Omega$ onto $\partial \R^2_+ := \{ x \in \R^2 : x_1=0  \}$. This is, for instance, possible by using boundary normal coordinates (see, for instance, \cite[Chapter 3.3]{FSU23}, \cite[Proposition 11.2.2.]{PSU23}, \cite[Section 2.1]{KKL01} or \cite[Lemma 3.15]{S13} for references on this).
We split our argument into two parts. First in Section \ref{sec:normal_coord} we recall the properties of boundary normal coordinates for the convenience of the reader. In Section \ref{sec:gen_domains_proofs} we then present the proofs of Theorems \ref{prop:gen_domain_nonlinear} and \ref{prop:gen_domain_lin}.

\subsubsection{Inverse boundary normal coordinates}
\label{sec:normal_coord}

Consider a bounded, connected set $D\subset\R^2$ with Lipschitz boundary and $\p D\ni0$.
Assume also that there exist $h\in C^{2,1}(\R;\R)$ with $h(0)=h'(0)=0$ and a sufficiently small constant $\rho>0$ such that
$$
D\cap Q_{2\rho}=\{(x,y) : y\in(-2\rho,2\rho), x>h(y)\},
$$
with the notation $Q_\ell:=\{(x,y) : |x|<\ell, |y|<\ell\}$.

We denote by $\nu_p$, $\tau_p$ and $\kappa_p$ the outer unit normal to $\p D$, the unit tangent (in counterclockwise orientation) to $\p D$ and the curvature of $\p D$ in $p\in\p D$, respectively. In particular, the assumptions on $h$ yield that $\nu_0=-e_1$ and $\tau_0=-e_2$.
According to the parametrization $y\mapsto(h(y),y)$ we recall that
$$
\nu_{(h(y),y)}=\frac{(-1,h'(y))}{\sqrt{1+(h'(y))^2}},
\quad
\tau_{(h(y),y)}=\frac{(-h'(y),-1)}{\sqrt{1+(h'(y))^2}},
\quad
\kappa_{(h(y),y)}=\frac{h''(y)}{1+(h'(y))^2} \,.
$$
We extend these three quantities onto $Q_\rho$ by denoting (with a slight abuse of notation) $\nu(y):=\nu_{(h(y),y)}$, $\tau(y):=\tau_{(h(y),y)}$ and $\kappa(y):=\kappa_{(h(y),y)}$.

We define the map $\Phi:Q_\rho\to\R^2$ given by
$$
\Phi(x,y)=(h(y),y)-x \, \nu(y).
$$
Notice that, by the condition $h(0)=h'(0)=0$ and by regularity, up to reducing the value of $\rho$, we can assume that $\Phi(Q_\rho)\subset Q_{2\rho}$.

The gradient of $\Phi$ reads
\begin{equation}\label{eq:jac-phi}
\nabla\Phi(x,y)=
\begin{pmatrix}
	-\nu_1(y) & h'(y)-x \, \p_y\nu_1(y) \\
	-\nu_2(y) & 1-x \, \p_y\nu_2(y)
\end{pmatrix}
=
\begin{pmatrix}
	-\nu_1(y) & \tau_1(y)(x\kappa(y)-\sqrt{1+(h'(y))^2}) \\
	-\nu_2(y) & \tau_2(y)(x\kappa(y)-\sqrt{1+(h'(y))^2})
\end{pmatrix}.
\end{equation}
In particular $\nabla\Phi(0,0)=Id$, so by the inverse function theorem, there exist two neighbourhoods of $0$, $U,V\subset\R^2$ such that $\Phi:U\to V$ is a $C^{1,1}$ diffeomorphism.

By construction, we additionally have that $\Phi(0)=0$ and that
$$
\Phi(U\cap\{(x,y)\in\R^2:x>0\})=D\cap V,
\quad
\Phi(U\cap\{(x,y)\in\R^2:x=0\})=\p D\cap V.
$$

\bigskip

We are now in position to define the diffeomorphism $F$ that will be used in the proofs of Theorems \ref{prop:gen_domain_nonlinear} and \ref{prop:gen_domain_lin}.
Here we will use the notation $\nu_{\p\Omega}(p)$ and $\tau_{\p\Omega}(p)$ to denote the unit normal vector and the tangent vector (in the counterclockwise orientation) to $\p\Omega$ in $p\in\p\Omega$.
Moreover, given $z\in\Omega$, we denote by $p_z\in\p\Omega$ the projection of $z$ onto $\p\Omega$, when it is single-valued.

\begin{lem}\label{lem:flat-F}
Let $\Omega$ be a Lipschitz, bounded, connected domain and let $\Gamma\subset\p\Omega$ be relatively open and $C^{2,1}$ regular and let $p_0\in\Gamma$.
Then, there exists a constant $r_0=r_0(\Gamma,\Omega)>0$ 
such that the following holds.
For every $0<r<r_0$ there exist a neighbourhood $\omega\ni p_0$ and a $C^{1,1}$-diffeomorphism $F:\omega\to Q_r$ such that $F(p_0)=(0,0)$,
\begin{equation}\label{eq:flat-F}
F(\omega\cap\Omega)=Q_r\cap\{(x,y)\in\R^2 : x>0\},
\quad
F(\omega\cap\Gamma)=Q_r\cap\{(x,y)\in\R^2 : x=0\},
\end{equation}
and such that $\nabla F(p_0)=R\in SO(2)$, with $R$ being the rotation such that $R\nu_{\p\Omega}(p_0)=-e_1$, $R\tau_{\p\Omega}(p_0)=-e_2$.

Moreover, there exist a constant $C=C(\Gamma)$ and a Lipschitz function $c:\omega\to\R$ such that
\begin{equation}\label{eq:F-cont}
\|\nabla F-R\|_{L^\infty(\omega)} \le Cr,
\end{equation}
and for every $z\in\omega$
\begin{equation}\label{eq:tangent-F}
-c(z)\tau_{\p\Omega}(p_z)=(\nabla F(z))^{-1}e_2
\quad \text{and} \quad
\|c(z)-1\|_{L^\infty(\omega)}\le (|\kappa(p_0)|+Cr)r.
\end{equation}
\end{lem}
\begin{proof}
Given $R$ as in the statement, let $D:=R(\Omega-p_0)$ and let $\Phi:U\to V$ be defined as above.
Let $r_0:=\sup\{\ell>0 : Q_\ell\subset U\}$.
Since $U$ is a neighbourhood of $0$, $r_0=r_0(\Gamma,\Omega)>0$.

Given $r$ as in the statement, we define $\omega:=R^t\Phi(Q_r)+p_0$ and $F:\omega\to Q_r$ as
$$
F(z):=\Phi^{-1}_{|Q_r}(R(z-p_0)), \quad \text{for every } r<r_0.
$$
It now remains to prove that this map satisfies the properties claimed in the statement.

The facts that $F$ is a $C^{1,1}$-diffeomerphism, $p_0\in\omega$, $F(p_0)=(0,0)$, \eqref{eq:flat-F}, that $\nabla F(p_0)=R$ and \eqref{eq:F-cont} are immediate by definition of $F$ and the properties of $\Phi$.

Lastly, denoting $(x,y)=F(z)$ and given that $(\nabla F(z))^{-1}=R^t\nabla\Phi(x,y)$, by \eqref{eq:jac-phi} we infer that
$$
(\nabla F(z))^{-1}e_2 = -\tilde c(x,y) R^t \tau(x,y) = \tilde c(F^{-1}(z)) \tau_{\p\Omega}(p_z).
$$
where we defined $\tilde c(x,y):=\sqrt{1+(h'(y))^2}-x\kappa(y)$ and we have used the fact that $\tau(x,y)=R \tau_{\p\Omega}(p_z)$, which follows by the definition of $D$ and of $\tau$.
Denoting $c:=\tilde c \circ F^{-1}$, \eqref{eq:tangent-F} follows since $\tilde c(0)=1$, it is Lipschitz continuous and $|\nabla\tilde c(x,y)|\le |\kappa(y)|+Cr$.
\end{proof}

In concluding this section, we collect some immediate consequences of Lemma \ref{lem:flat-F} in the following remark. We will use these in the subsequent section.

\begin{rmk}
We remark that the following more precise results hold:
\begin{itemize}
\item[(i)] As can be seen from the proof of Lemma \ref{lem:flat-F}, the maps $F$ depending on $r$ are obtained as restriction of the same map $F:\omega_0\to Q_{r_0}$ on the sets $\omega$, where we set $\omega_0:=R^t\Phi(Q_{r_0})+p_0$ for $r_0>0$ as in the proof of Lemma \ref{lem:flat-F} and recall that $\omega:=R^t\Phi(Q_r)+p_0$ for $r\in (0,r_0)$. 
\item[(ii)] If $\Gamma$ is a segment, then we can choose $F$ to be a rigid motion and \eqref{eq:tangent-F} holds true for $c\equiv 1$.
\item[(iii)] By expanding the determinant and by the properties of the Frobenius norm, direct implications of \eqref{eq:F-cont} are that
\begin{equation}\label{eq:jacobian}
\|\det(\nabla F)-1\|_{L^\infty(\omega)}\le Cr,
\quad
\|\det(\nabla F^{-1})-1\|_{L^\infty(Q_r)}\le Cr
\end{equation}
and that
\begin{equation}\label{eq:norm-inv}
\|\nabla F^{-1}-R^t\|_{L^\infty(Q_r)}\le Cr
\end{equation}
by possibly enlarging the constant $C$, e.g., by multiplying it by a universal constant.
\end{itemize}
\end{rmk}

\subsubsection{Proofs of Theorems \ref{prop:gen_domain_nonlinear} and \ref{prop:gen_domain_lin}}
\label{sec:gen_domains_proofs}

In this section we make use of the boundary normal coordinates discussed in the previous section and present the proofs of Theorems \ref{prop:gen_domain_nonlinear} and \ref{prop:gen_domain_lin}. In order to obtain the desired extension to arbitrary domains for the geometrically nonlinear setting, we directly reduce the argument to the normalized setting by transforming a sufficiently small coordinate patch into the unit square and invoking Proposition \ref{prop:ref_domain}(a). The proof of Theorem \ref{prop:gen_domain_lin} is slightly more involved. Due to the more complicated transformation behaviour of the symmetrized gradient, we first invoke Korn's inequality to obtain a gradient bound and then adapt the proof of Proposition \ref{prop:ref_domain}(b) into this setting.

\begin{proof}[Proof of Theorem \ref{prop:gen_domain_nonlinear}]
The statement follows directly from the result for the square domain after a change of variables using boundary normal coordinates (locally) mapping the boundary $\Gamma$ to the line $\{x=0\} \subset \R^2$. More precisely, this follows by the transformation behaviour of the gradient and the tangent vectors to $\Gamma$.

Indeed, let $p_0\in\Gamma$.
For a small parameter $0<r<r_0(\Gamma,\Omega)$ to be fixed later let $\omega\subset\Omega$ and $F:\omega\to Q_r$ be given by Lemma \ref{lem:flat-F}.

With this notation fixed, let $u: \Omega \rightarrow \R^2$ be as in the statement of Theorem \ref{prop:gen_domain_nonlinear} 
and let $v:Q_r\to\R^2$ be defined as $v(y):=R \, u(F^{-1}(y))$, where $R=\nabla F(p_0)$. Then,
\begin{align}
\label{eq:trafo_domain_elast}
\begin{split}
\int\limits_{\Omega} \dist^2(\nabla u(x) , \K) dx
& \geq \int\limits_{\omega} \dist^2(\nabla u(x) , \K) dx
 = \int\limits_{\omega} \dist^2((\nabla u( F^{-1}(F(x))), \K) dx \\
&= \int\limits_{Q_r}\dist^2(\nabla u|_{F^{-1}(y)},\K) |\det(\nabla F^{-1}(y))| dy \\
&=  \int\limits_{Q_r}\dist^2(R^{t} \nabla_y v(y) (\nabla F|_{F^{-1}(y)}),\K) |\det(\nabla F^{-1}(y))| dy\\
&\ge \frac{1}{2}\int\limits_{Q_r}\dist^2(\nabla v(y),R\K(\nabla F|_{F^{-1}(y)})^{-1}) |\det(\nabla F^{-1}(y))| dy \\
& = \frac{1}{2}\int\limits_{Q_r}\dist^2(\nabla v(y),\hat{\mathcal{K}}(y)) |\det(\nabla F^{-1}(y))| dy,
\end{split}
\end{align}
where $\hat{\mathcal{K}}(y):=R\K(\nabla F|_{F^{-1}(y)})^{-1}$ and $r>0$ is small enough.
Let us now write  $\hat{\mathcal{K}}:=\bigcup_{y\in Q_r}\hat{\mathcal{K}}(y)$ and $\tilde v(z):= r^{-1} v_{|(0,r)^2}(rz)$.
We will apply Proposition \ref{prop:gen_domain_nonlinear} to $\tilde{v}$ and $\hat \K$.
First, it follows by the estimate \eqref{eq:trafo_domain_elast} above and \eqref{eq:jacobian} that there exists $c_{\Gamma}>0$ a constant depending on $\Gamma$ such that 
\begin{align}\label{eq:elastic-F}
\int\limits_{\Omega} \dist^2(\nabla u , \K) dx
\geq c_{\Gamma} r^2 \int\limits_{(0,1)^2} \dist^2(\nabla \tilde{v}, \hat\K) dx.
\end{align}
Next, using the transformation behaviour of tangent vectors \eqref{eq:tangent-F}, assumption \eqref{eq:dist} turns into 
\begin{align*}
\min\limits_{\hat{K} \in \hat{\mathcal{K}}} |\hat{K}e_2 - e_2| 
= \min\limits_{x\in F^{-1}(Q_r),\, K \in \mathcal{K}}|c(x)K \tau(p_x)-\tau(p_0)|,
\end{align*}
Moreover, for every $x\in \omega$ by \eqref{eq:norm-inv} and \eqref{eq:tangent-F} we have
\begin{align*}
	|c(x)K\tau(p_x)-\tau(p_0)| &\ge \big|c(x)|K\tau(p_x)-\tau(p_x)|-\|(\nabla F)^{-1}-R^t\|_{L^\infty(Q_r)}\big| \\
	&\ge (1-Cr)|K\tau(p_x)-\tau(p_x)|- Cr \ge \frac{d}{2},
\end{align*}
up to reducing the size of $r$. 
Thus we have recovered the lower bound control for the set $\hat{\mathcal{K}}$ assumed in Proposition \ref{prop:ref_domain} ($1$) for $\frac{d}{2}$ (in place of $d$).

To apply Proposition \ref{prop:ref_domain} (with $d'=\frac{d}{2}$ in place of $d$) we further need to check the boundary condition \eqref{eq:trace_nonlin}.
Let $I\subset(0,1)$ be an open interval and let $\Gamma':=\{F^{-1}(0,rs): s\in I\}\subset\Gamma$, which is relatively open and connected by the regularity of $F$.
By definition of $\tilde v$, for every $s\in I$ we have $\tilde v(0,s)=r^{-1}Ru(F^{-1}(0,rs))$.
Writing $x=F^{-1}(0,rs)$ we also have $(0,s)=r^{-1}F(x).$ From \eqref{eq:tangent-F} and by definition of $R$, for every $s_1,s_2\in I$ we have
$$
|(\tilde v(0,s_1)-(0,s_1))-(\tilde v(0,s_2)-(0,s_2))|=r^{-1}|(Ru(x_1)-F(x_1))-(Ru(x_2)-F(x_2))|,
$$
where we denoted $x_i=F^{-1}(0,rs_i).$
By \eqref{eq:F-cont} we have that
$$
|(F(x_1)-Rx_1)-(F(x_2)-Rx_2)|\le Cr|x_1-x_2|, 
$$
hence, the triangle inequality yields
$$
|(\tilde v(0,s_1)-(0,s_1))-(\tilde v(0,s_2)-(0,s_2))|\le r^{-1}|(u(x_1)-x_1)-(u(x_2)-x_2)|+C|x_1-x_2|.
$$
Noticing that (up to reducing the value of $r$) $|x_1-x_2|\le 2\mathcal{H}^1(\Gamma')$ and that, by the area formula and \eqref{eq:tangent-F} $\big|\mathcal{H}^1(\Gamma')-r|I|\big|\le 2Cr^2|I|$, condition \eqref{eq:bc-nonlin-gen} implies that
\begin{align*}
\esup_{s,s'\in I}|(\tilde v(0,s)-(0,s))-(\tilde v(0,s')-(0,s'))|
& \le\Big(\frac{d}{208}r^{-1}+2C\Big)\mathcal{H}^{1}(\Gamma')\\
& \le\frac{d}{208}(1+C_1r)|I|,
\end{align*}
for some $C_1>0$ depending on $C$ and $d$.
Thus,
the boundary condition for $\tilde v$ follows for $r$ sufficiently small.

We now deal with the surface energy contribution.
By definition of $v$ and by the chain rule we have $\nabla v(y)=R\nabla u_{|F^{-1}(y)}\nabla F^{-1}(y)$.

We first observe that the minimization problem for $E_{\epsilon}$ can be reduced to functions such that $\|\nabla u\|_{L^2(\Omega)}\le M_0$ for some $M_0>0$ depending only on $\Omega$ and $\K$.
Indeed, since $E_\epsilon(0)\le |\Omega|\max_{M\in\K}|M|^2:=M_1$ we can reduce to functions $u$ such that $E_\epsilon(u)\le M_1$. Hence, integrating the pointwise estimate
$$
|\nabla u(x)|^2 \le 2\dist^2(\nabla u(x),\K)+2\max_{M\in\K}|M|^2, \quad \text{for almost every } x\in\Omega,
$$
we get that $\|\nabla u\|_{L^2(\Omega)}\le 2M_1+2\max_{M\in\K}|M|^2:=M_0$.

By approximation (see e.g. \cite[Theorem 3.9]{AFP2000}) $\nabla v\in BV(Q_r;\R^{2\times2})$ and 
$$
|D^2v|(Q_r)\le C_1(|D^2u|(\omega)+\|\nabla u\|_{L^1(\omega)}).
$$
which yields, after a rescaling and an application of H{\"o}lder's inequality, that
\begin{equation}\label{eq:surface-F}
|D^2\tilde v|((0,1)^2)\le \frac{1}{r} C_1(|D^2u|(\omega)+r),
\end{equation}
for some $C_1>0$ and $r$ sufficiently small but fixed.

Combining \eqref{eq:elastic-F} and \eqref{eq:surface-F} and choosing $r>0$ sufficiently small, we eventually obtain
$$
E_\epsilon(u) \ge c_\Gamma r^2 \int_{(0,1)^2}\dist^2(\nabla\tilde v,\hat\K)dy + C_2 r \ \epsilon|D^2\tv|((0,1)^2) - r \ \epsilon,
$$
for some $C_2>0$ depending on $\Omega$ and $\K$.

Now, let $0< \epsilon_0 < 1$ be small enough such that for all $0 < \epsilon \leq \epsilon_0$ it holds
\[
\min\{  C_2 r \epsilon (|\log C_2r\epsilon| + 1 ), 1 \} - r \epsilon \geq C_3 \min\{ \epsilon  (|\log \epsilon| + 1 ), 1 \},
\]
with $C_3=\frac{1}{2}\min\{C_2r|\log(C_2r)|,1\}$.
Applying Proposition \ref{prop:ref_domain} for $0< \epsilon \leq \epsilon_0 < 1$ we then obtain for a universal $C_0 > 0$
\begin{align*}
E_\epsilon(u) & \ge c_\Gamma r^2 \int_{(0,1)^2}\dist^2(\nabla\tilde v,\hat\K)dy + C_2 r \ \epsilon|D^2\tv|((0,1)^2) - r \ \epsilon \\ 
& \ge C_0 \min\{\epsilon C_2r(|\log\epsilon C_2 r|+1),1\} - r \epsilon \geq  C_0 C_3 \min\{\epsilon(|\log\epsilon |+1),1\},
\end{align*}
up to reducing the value of $\epsilon_0$ if needed.
On the other hand, for $\epsilon > \epsilon_0$ it follows for a universal $C_0' > 0$
\[
E_{\epsilon}(u) \geq E_{\epsilon_0}(u) \geq C_0C_3 \min\{\epsilon_0 (|\log\epsilon_0|+1),1\} \geq C_0' \min\{\epsilon(|\log\epsilon|+1),1\}.
\]
\end{proof}

The proof of Theorem \ref{prop:gen_domain_lin} provided below is an adaptation of the arguments given for the reference domain (cf.\ Proposition \ref{prop:ref_domain}(2)).
Following the strategy of the geometrically nonlinear setting in which one replaces Euclidean coordinates with boundary normal coordinates gives rise to an additional difficulty: since we are controlling the gradient along curves (instead of straight lines) we need to control the full gradient in terms of the energy $F_\epsilon$.
This can be achieved by the following Korn-type inequality, see, e.g.,~\cite{Pompe}.

\begin{prop}\label{prop: Korn}
    Let $U \subseteq \R^d $ be an open, connected and bounded set with Lipschitz boundary. Moreoever, let $\Gamma \subseteq \partial \Omega$ with $\mathcal{H}^{d-1}(\Gamma) > 0$. Then there exists a constant $C>0$ such that for every $v \in W^{1,2}(U;\R^d)$ it holds
    \[
    \int_{U} |\nabla v|^2 \, dx \leq C \left( \int_U |e(v)|^2 \, dx + \int_{\Gamma} |v|^2 \, d \mathcal{H}^{d-1} \right).
    \]
\end{prop}
\begin{rmk}
    Given $U \subseteq \R^d$ as above define for $h>0$ the sets $U_h = h U$, and $\Gamma_h = h \Gamma$. Note that by scaling there exists $C>0$ such that it holds for all $h > 0$ and $v \in W^{1,2}(U_h; \R^d)$ that
    \[
    \int_{U_h} |\nabla v|^2 \, dx \leq C \left( \int_{U_h} |e(v)|^2 \, dx + \frac1h \int_{\Gamma_h} |v|^2 \right).
    \]
\end{rmk}

\begin{proof}[Proof of Theorem \ref{prop:gen_domain_lin}]
We proceed with a similar change of coordinates as that in the geometrically nonlinear case. 
Let $p_0\in\Gamma$.
For a small parameter $0<r<r_0(\Gamma,\Omega)$ to be fixed later, let $\omega\subset\Omega$ and $F:\omega\to Q_r$ be given by Lemma \ref{lem:flat-F}.
Let $u:\Omega\to\R^2$ be as in the statement of Theorem \ref{prop:gen_domain_lin}.
We denote by $v:Q_r \to \R^2$ the function defined as $v(y):=R(u(F^{-1}(y)))$.
By the chain rule we have
\begin{align*}
	\nabla v(y) &= R\nabla u|_{F^{-1}(y)}\nabla F^{-1}(y).
\end{align*}
As a final adjustment, we define $\tilde v(z):=r^{-1}v_{|(0,r)^2}(rz)$, so that $\tilde v$ is defined in the reference square $(0,1)^2$.
For notational convenience we set
$$
\tilde F_\epsilon(\tilde v):=\int_{(0,1)^2} \dist^2((R^t\nabla\tilde v(z)\nabla F_{|F^{-1}(rz)})^{sym},\tilde\K)dz+\epsilon |D^2\tilde v|((0,1)^2),
$$
and extend this notation to its corresponding localizations.

We now argue in three steps.
First, we will provide a lower bound control of the type $F_\epsilon(u) \ge c_\Gamma \tilde F_\epsilon(\tilde v) - C_{\Gamma} \epsilon$.
Secondly, we will retrace the proof of Proposition \ref{prop:ref_domain}(2), for the (more general) energy $\tilde F_\epsilon$, which will then, in the last step, provide the desired lower scaling bound using an argument similar to the one at the end of the proof of Theorem \ref{prop:gen_domain_nonlinear}.

\smallskip

\emph{Step 1: change of variables.} We notice that
\begin{equation*}
\begin{split}
\int_\Omega \dist^2(e(u(x)),\tilde \K)dx &\ge \int_{\omega} \dist^2(e(u(x)),\tilde\K)dx \\
&= \int_{\omega} \dist^2((R^t\nabla v_{|F(x)}\nabla F(x))^{sym},\tilde\K)dx \\
&= \int_{Q_r} \dist^2((R^t\nabla v(y)\nabla F_{|F^{-1}(y)})^{sym},\tilde\K)|\det(\nabla F^{-1}(y))|dy \\
&\ge \frac{1}{2}\int_{Q_r} \dist^2((R^t\nabla v(y)\nabla F_{|F^{-1}(y)})^{sym},\tilde\K)dy,
\end{split}
\end{equation*}
with the last inequality being a consequence of \eqref{eq:jacobian} for $r>0$ sufficiently small.
By a scaling argument we also get
\begin{equation}\label{eq:lin1proof}
\int_\Omega \dist^2(e(u(x)), \tilde \K)dx \ge \frac{r^2}{2}\int_{(0,1)^2} \dist^2((R^t\nabla\tilde v(z)\nabla F_{|F^{-1}(rz)})^{sym},\tilde\K)\,dz.
\end{equation}
Gathering \eqref{eq:lin1proof} and \eqref{eq:surface-F} we obtain
$$
F_\epsilon(u) \ge  c_\Gamma \tilde F_\epsilon(\tilde v)- C_{\Gamma} \epsilon,
$$
for some $C_{\Gamma}, c_\Gamma>0$.

Additionally, by arguing analogously as in the proof of Theorem \ref{prop:gen_domain_nonlinear}, the boundary condition \eqref{eq:bc-lin-gen} turns into the following estimate for $\tilde v$ 
\begin{align}
\label{eq:bc-change-var}
\begin{split}
\esup_{s_1,s_2\in\tilde I}|\tilde v(0,s_1)-\tilde v(0,s_2)|
& \le\frac{d}{64\cdot17\cdot9\cdot\sqrt{2}}r^{-1}\mathcal{H}^1(\tilde \Gamma)\le\Big(\frac{d}{64\cdot17\cdot9\cdot\sqrt{2}}+C_1r\Big)|\tilde I|\\
& \le\frac{d}{32\cdot17\cdot9\cdot\sqrt{2}}|\tilde I|
\end{split}
\end{align}
for every open interval $\tilde I\subset(0,1)$ up to reducing the value of $r$. Here we have used the notation $\tilde\Gamma:=\{F^{-1}(0,rs)\in\Gamma : s\in\tilde I\}$ and $C_1>0$ is a constant depending on $\Gamma$.
In the following we will prove a lower bound of the form $\tilde{F}_{\epsilon}(\tilde v) \geq c \min\{ \epsilon(|\log \epsilon| + 1) ,1\}$. We start again with a version of the lower bound on slices $\{t\} \times (0,1)$. 
\smallskip

\emph{Step 2: lower bound for $\tilde F_\epsilon$.}
We consider the working assumption that $\tilde F_\epsilon(\tv)\le c_0\min\{ \epsilon(|\log \epsilon| + 1) ,1\}$ for some $c_0>0$.
This can always be assumed without loss of generality since otherwise there is nothing to prove.
We define the constants $c_2:=\frac{d}{17}$, $c_3:=2\max_{\tilde{K}\in\tilde\K}|\tilde{K}|$, $c_1=\frac{1}{2}+16\frac{c_3}{c_2}$ and $c_4=\frac{c_2 c_1^2}{32\sqrt{2}(c_1+1)^2}$.
We notice that $c_4\ge \frac{d}{32\cdot17\cdot9\cdot\sqrt{2}}$.
The remainder of this step is devoted to proving the following claim.

\bigskip

\underline{\textit{Claim:}} There exists $c>0$ such that for all $t \in (\frac{1}{4c_1} \min\{ 1, \epsilon (|\log \epsilon| + 1) \},\frac1{2c_1})$, $\tv\in W^{1,2}((0,1)^2;\R^2)$ satisfying \eqref{eq:bc-change-var}  and all $\epsilon > 0$  at least one of the following conditions holds
\begin{equation}\label{eq: est claim rotated}
\tilde F_\epsilon(\tv;\{t\} \times (0,1)) \geq c \min\Big\{ \frac{\epsilon}t, 1 \Big\} \text{ or } |D^2 \tv|((t/2,t) \times (0,1)) \geq c.
\end{equation}

\begin{figure}
\includegraphics{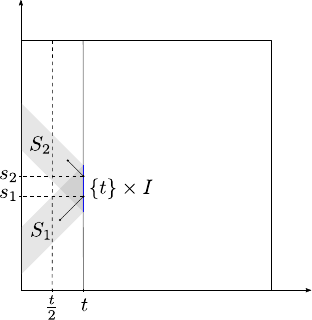}
\caption{The figure depicts the interval  $\{t\}\times I$ (in blue), the stripes in which the energy is controlled $S_1, S_2$ and the points in which the gradient is controlled $(t,s_1), (t,s_2)$.}
\label{fig:slices}
\end{figure}

\bigskip
Let $t\in(\frac{1}{4c_1}\min\{1,\epsilon(|\log\epsilon|+1)\},\frac1{2c_1})$. 
Moreover, we may further assume that $|D^2 \tv|((t/2,t) \times (0,1)) \leq \frac1{40}$.
Let $I\subset(\frac{1}{4},\frac{3}{4})$ be an interval chosen such that $|I|=c_1t$ and for $i=1,2$
\begin{equation}\label{eq: choice slice}
\tilde F_\epsilon(\tv;\{t\}\times I) \le 16|I|\tilde F_\epsilon(\tv;\{t\}\times(0,1)), \quad
\tilde F_\epsilon(\tv;S_i) \le 16|I|\tilde F_\epsilon(\tv) \quad \text{ and } \quad |D^2 \tv|(R_i) \leq  |I|, 
\end{equation}
where $S_1:=\{(t-\xi,s-\xi):s\in I, 0<\xi<t\}$, $S_2:=\{(t-\xi,s+\xi):s\in I, 0<\xi<t\}$  and  $R_i := S_i\cap \left((\frac{t}{2},t)\times(0,1) \right)$, see Figure \ref{fig:slices}.

Let $s_{0,i} := \min\{\eta>0 : (\xi,\eta)\in S_i\}$. By Korn's inequality with boundary conditions, Proposition \ref{prop: Korn}, for $U=t^{-1} \left(S_i - (0,s_{0,i})\right)\subset (0,1)\times(0,c_1+1)$ and $h=t$ we obtain
\begin{align}
\label{eq:Korn1}
\int_{S_i}|\nabla\tv(y)|^2dy \le C_0\Big(\int_{S_i}|(R^t\nabla\tv(y) R)^{sym}|^2dy+t^{-1}\int_{I_i}\Big|\tv(0,s)-\fint_{I_i}\tv(0,\cdot)\Big|^2ds\Big)
\end{align}
with $I_i:=\partial S_i\cap\big(\{0\}\times(0,1)\big)$. 
We stress the fact that, since $t^{-1} \left(S_i - (0,s_{0,i})\right)$ is a fixed domain (independent of $t$) the constant $C_0$ above depends only on the constant $c_1$ and on $F$, hence only on $\Gamma$, $\Omega$ and $\tilde \K$.
By summing and subtracting the term $(R^t\nabla\tv(y)\nabla F_{|F^{-1}(ry)})^{sym}$ inside the modulus in the first term of the right-hand side of \eqref{eq:Korn1} we obtain
\begin{align}
\label{eq:Korn_apply}
\begin{split}
&\int_{S_i}|\nabla\tv(y)|^2dy \le C_0\Big(\int_{S_i}|(R^t\nabla\tv(y)\nabF)^{sym}|^2dy \\
& \quad+ \int_{S_i}|R^t\nabla\tv(y)(R-\nabF)|^2dy + t^{-1}\int_{I_i}\Big|\tv(0,s)-\fint_{I_i}\tv(0,\cdot)\Big|^2ds\Big).
\end{split}
\end{align}
Noticing that $ry\in Q_r$ and that $R:= \nabla F(p_0)$ with $p_0 = F(0)$, since $\nabla F\in C^{0,1}(\omega)$ (with $\omega \subset \Omega$ as in Lemma \ref{lem:flat-F}) we have that $|R-\nabF| = |\nabla F|_{F^{-1}(0)}- \nabF|\le Cr$ for every $y\in S_i$ and $i\in\{1,2\}$.
Thus, up to possibly reducing the value of $r$, we can absorb the first term in the second line of \eqref{eq:Korn_apply} into the left-hand side of \eqref{eq:Korn_apply}.
This leads to
$$
\int_{S_i}|\nabla\tv(y)|^2dy \le 2C_0\Big(\int_{S_i}|(R^t\nabla\tv(y)\nabF)^{sym}|^2dy+t^{-1}\int_{I_i}\Big|\tv(0,s)-\fint_{I_i}\tv(0,\cdot)\Big|^2ds\Big).
$$
Let $\tilde K\in L^\infty(S_i;\tilde\K)$ be such that
$$
|(R^t\nabla\tv(y)\nabF)^{sym}-\tilde K(y)|=\dist((R^t\nabla\tv(y)\nabF)^{sym},\tilde\K).
$$
Recalling that $c_4\ge\frac{d}{32\cdot17\cdot9\cdot\sqrt{2}}$, \eqref{eq:bc-change-var} and the fact that $\tilde F_\epsilon(\tv)\le c_0 4 c_1 t=4c_0|I|$ we infer 
\begin{equation}\label{eq:L2-v}
\begin{split}
\int_{S_i}|\nabla\tv(y)|^2dy &\le 2C_0\Big(2\int_{S_i}|(R^t\nabla\tv(y)\nabF)^{sym}-\tilde K(y)|^2dy \\
& \qquad\qquad\qquad +2\int_{S_i}|\tilde K(y)|^2dy+t^{-1}\int_{I_i}\Big|\tv(0,s)-\fint_{I_i}\tv(0,\cdot)\Big|^2ds\Big) \\
&\le 2C_0\Big(32|I|\tilde F_\epsilon(v)+2c_1^{-1}c_3^2|I|^2+c_1c_4^2|I|^2\Big) \\
&\le 2C_0(128 c_0+2c_1^{-1}c_3^2+c_1c_4^2)|I|^2=: d_1 |I|^2.
\end{split}
\end{equation}
Thus, using \eqref{eq: choice slice} there exist $(\tilde t_i,\tilde s_i) \in S_i$ with $\tilde t_i > t/2$, such that 
\begin{equation}\label{eq:grad-contr}
|\nabla \tv(\tilde t_i,\tilde s_i)| \le (3 c_1 d_1)^\frac{1}{2}  \quad \text{ and } \quad | (\nabla \tv((0,\tilde s_i+(-1)^i\tilde t_i) + \cdot (1,1))'|((t/2,t)) \leq 3.  
\end{equation}
Denoting $s_i:=\tilde s_i+(-1)^i\tilde t_i+t$ 
it follows that 
\begin{align}
\label{eq:grad_comb}
|\nabla \tv(t,s_i)| \leq (3c_1d_1)^{1/2} + 3,
\end{align}
and by definition of $S_i$ it holds $(t,s_i) \in \{t\} \times I$, see again Figure \ref{fig:slices}.

Analogously as in the proof of Proposition \ref{prop:ref_domain}(2), it suffices to prove the claim under the smallness assumptions that 
\begin{align}
\label{eq:small}
|\p_2\nabla\tv(t,\cdot)|(I)<c_2 \mbox{ and } \dist((R^t\nabla tv(y)\nabla F_{|F^{-1}(ry)})^{sym},\tilde\K)\le 3c_2 \text{ for } \mathcal{H}^1\text{-a.e. }y\in\{t\}\times I.
\end{align}
We first consider the case $i=1$ with the other case being completely analogous.
By condition \eqref{eq:dist-lin} and property \eqref{eq:tangent-F}, setting $y_s:=(t,s_1)$ and $x_s:=F^{-1}(ry_s)$, we obtain
\begin{align*}
d &\le |\tau(p_{x_s})\cdot\tilde K(y_s)\tau(p_{x_s})| \\
&\le |\tau(p_{x_s})\cdot(\tilde K(y_s)-(R^t\nabla\tv(y_s)\nabla F(x_s)))\tau(p_{x_s})| + |\tau(p_{x_s})\cdot(R^t\nabla\tv(y_s)\nabla F(x_s))\tau(p_{x_s})| \\
&\leq \dist((R^t\nabla\tv(y_s)\nabla F(x_s))^{sym},\tilde\K)+c(x_s)^{-1}|(R\tau(p_{x_s}))\cdot\p_2 \tv(y_s)|.
\end{align*}

Adding and subtracting (inside the modulus of the last term on the right-hand side above) $(\p_2\tv_2(y_s)-\p_2\tv_2(y))$ for a generic $y\in\{t\}\times I$, denoting $x:=F^{-1}(ry)$ and recalling \eqref{eq:grad_comb}, up to reducing the value of $r$, we arrive at
\begin{align*}
\frac{d}{2} &\le \frac{1}{2}\dist((R^t\nabla\tv(y_s)\nabla F(x_s))^{sym},\tilde\K) \\
& \qquad +|((R-\nabla F(x_s))\tau(p_{x_s}))\cdot\p_2\tv(y_s)|+|\p_2\tv_2(y_s)-\p_2\tv_2(y)|+|\p_2\tv_2(y)| \\
&\le \frac{3}{2}c_2+|\nabla\tv(y_s)|\|R-\nabla F\|_{L^\infty(\omega)}+c_2+|\p_2\tv_2(y)| \\
&\le \frac{5}{2}c_2+((3c_1d_1)^\frac{1}{2}+3)Cr+|\p_2\tv_2(y)|.
\end{align*}
By choosing $r$ small enough and by the choice of the constants we conclude that $|\p_2\tv_2(y)|\ge5c_2$ for every $y\in\{t\}\times I$.
As, by assumption, $|\p_2\nabla\tv(t,\cdot)|(I)<c_2$, we may assume without loss of generality that $\p_2\tv_2(y)\ge 5c_2$.

Let $\zeta=(1,1)$, $s_0\in I$ and denote $\zeta(y):=\nu(p_x)+c(x)\tau(p_x)$.
By \eqref{eq:bc-change-var} we infer
\begin{align*}
\int_I |(\tv(t,\eta)-\tv(0,s_0))\cdot\zeta|d\eta &\le \int_I|(\tv(t,\eta)-\tv(0,\eta-t))\cdot\zeta|d\eta+\int_I|(\tv(0,\eta-t)-\tv(0,s_0))\cdot\zeta|d\eta \\
&\le \int_{S_1}|\zeta\cdot\nabla\tv(y)\zeta|dy+\sqrt{2}c_4(|I|+t)^2 \\
&= \int_{S_1}|(\nabF\zeta(y))\cdot \nabla\tv(y)(\nabF\zeta(y))|dy+\sqrt{2}c_4(|I|+t)^2.
\end{align*}
Summing and subtracting $(R\zeta(y))\cdot\nabla\tv(y)(\nabF\zeta(y))+\zeta(y)\cdot\tilde K(y)\zeta(y)$ on the right-hand side and applying the triangle inequality, we obtain 
\begin{align*}
\int_I |(\tv(t,\eta)-\tv(0,s_0))\cdot\zeta|d\eta &\le \int_{S_1}|\zeta(y)\cdot \big((R^t\nabla\tv(y)\nabF)^{sym}-\tilde K(y)\big)\zeta(y)|dy \\
& \qquad +\int_{S_1}|\zeta(y)\cdot \big((\nabF-R)^t\nabla\tv(y)\nabF\big)\zeta(y)|dy \\
& \qquad +\int_{S_1}|\zeta(y)\cdot\tilde K(y)\zeta(y)|dy+\sqrt{2}c_4(|I|+t)^2 \\
& \le 4\int_{S_1}\dist((R^t\nabla\tv(y)\nabF)^{sym},\tilde\K)dy \\
& \qquad + 4\|\nabla F-R\|_{L^\infty(\omega)}\int_{S_1}|\nabla\tv(y)|dy + c_3|S_1|+\sqrt{2}c_4(|I|+t)^2 \\
&\le 4|S_1|^\frac{1}{2} \tilde F_\epsilon(\tv;S_1)^\frac{1}{2} \\ 
& \qquad +4 \|\nabla F-R\|_{L^\infty(\omega)} |S_1|^\frac{1}{2}\|\nabla\tv\|_{L^2(S_1)}+c_3|S_1|+\sqrt{2}c_4(|I|+t)^2,
\end{align*}
where we have used Hölder's inequality in the last step.
By \eqref{eq:L2-v}, recalling the choice of $I$, \eqref{eq: choice slice} and \eqref{eq:small}, we conclude that
\begin{align*}
\int_I |(\tv(t,\eta)-\tv(0,s_0))\cdot\zeta|d\eta &\le 16|S_1|^\frac{1}{2}|I|^\frac{1}{2} \tilde F_\epsilon(\tv)^\frac{1}{2}+4 \|\nabla F-R\|_{L^\infty(\omega)} |S_1|^\frac{1}{2}d_1^\frac{1}{2}|I| \\
& \qquad +c_3|S_1|+\sqrt{2}c_4(|I|+t)^2 \\
&\le 16 t^\frac{1}{2} |I| \tilde F_\epsilon(\tv)^\frac{1}{2}+4 Crc_1^{-\frac{1}{2}}d_1^\frac{1}{2}|I|^2+c_3t|I|+\sqrt{2}c_4(|I|+t)^2.
\end{align*}
Arguing analogously for $\zeta=(1,-1)$, we obtain
\begin{equation}\label{eq:ineq-gen1}
\int_I|\tv_2(t,\eta)-\tv_2(0,s_0)|d\eta \le 16 t^\frac{1}{2}|I|F_\epsilon(\tv)^\frac{1}{2}+c_3t|I|+(4Cr(c_1^{-1}d_1)^\frac{1}{2}+\sqrt{2}c_4)(|I|+t)^2.
\end{equation}

We now follow the strategy of the proof of Proposition \ref{prop:ref_domain}(2).
Let $w:I\to\R$ be defined such that $w'(y)=\tau(p_{x})\cdot\tilde K(y)\tau(p_{x})$, where we recall that $x=F^{-1}(ry)$.
By \eqref{eq:dist-lin}, \eqref{eq:tangent-F} and the fact that $\p_2 v_2\ge 5c_2$ on $\{t\}\times I$, by summing and subtracting $\p_2\tv_2(y)+\tau(p_x)\cdot(R^t\nabla\tv(y)\nabla F(x))\tau(p_x)$ we observe that
\begin{align*}
w'(y_2) 
&= \p_2\tv_2(y) - \tau(p_{x})\cdot(R^t\nabla\tv(y)\nabla F(x)-\tilde K(y))\tau(p_{x}) \\
& \qquad + c(x)\big(((\nabla F(x))^{-1}-R^t) e_2\big)\cdot R^t\nabla\tv(y)\nabla F(x)\tau(p_{x}) \\
&\ge \p_2\tv_2(y)-\dist((R^t\nabla\tv(y)\nabla F(x))^{sym},\tilde\K) \\
& \qquad + c(x)\big(((\nabla F(x))^{-1}-R^t) e_2\big)\cdot R^t\nabla\tv(y)\nabla F(x)\tau(p_{x}).
\end{align*}
To control the last term on the right-hand side above, we invoke the triangle inequality with the gradient evaluated on $y_s$ and exploit \eqref{eq:grad-contr}.
Hence, summing and subtracting $c(x)\big(((\nabla F(x))^{-1}-R^t) e_2\big)\cdot R^t\nabla\tv(y_s)\nabla F(x_s)\tau(p_{x_s})$ and using the smallness assumptions \eqref{eq:small}, we obtain
\begin{align}
\label{eq:affine1}
\begin{split}
w'(y_2) &\ge \p_2\tv_2(y)-\dist((R^t\nabla\tv(y)\nabla F(x))^{sym},\tilde\K) \\
& \qquad - c(x)\big(((\nabla F(x))^{-1}-R^t) e_2\big)\cdot R^t\nabla\tv(y_s)\nabla F(x_s)\tau(p_{x_s}) \\
& \qquad + c(x)\big(((\nabla F(x))^{-1}-R^t) e_2\big)\cdot\big(R^t\nabla\tv(y_s)\nabla F(x_s)\tau(p_{x_s})-R^t\nabla\tv(y)\nabla F(x)\tau(p_{x})\big) \\
&\ge \p_2\tv_2(y)-\dist((R^t\nabla\tv(y)\nabla F(x))^{sym},\tilde\K) \\
& \qquad - 2\|\nabla F^{-1}-R^t\|_{L^\infty(Q_r)} |\nabla\tv(y_s)| - 2\|\nabla F^{-1}-R^t\|_{L^\infty(Q_r)}|\nabla\tv(y_s)-\nabla\tv(y)| \\
&\ge 5c_2 - (3+2Cr)c_2 - ((8c_1 d_1)^{\frac{1}{2}}+3)2Cr \ge c_2,
\end{split}
\end{align}
up to reducing the value of $r$. 
Proceeding similarly as above, we also obtain
\begin{align}
\label{eq:affine2}
\begin{split}
|\p_2\tv_2(y)-w'(y_2)| &\le |\tau(p_x)\cdot(R^t\nabla  \tv(y)\nabla F(x)-\tilde K(y))\tau(p_x)| \\
& \qquad + 2\|\nabla F^{-1}-R^t\|_{L^\infty(Q_r)}(|\nabla\tv(y_s)|+|\nabla \tv(y_s)-\nabla \tv(y)|).
\end{split}
\end{align}
Following the strategy of the proof of Proposition \ref{prop:ref_domain}, from the two inequalities \eqref{eq:affine1}, \eqref{eq:affine2}, we infer  for $a_1 = \fint_I \tilde{v}_2(t,\cdot) - w \, ds$
\begin{equation}\label{eq:ineq-gen2}
\begin{split}
\frac{c_2}{4}|I|^2 &\le \|\tv_2(t,\cdot)-w-a_1\|_{L^1(I)}+\|\tv_2(t,\cdot)-\tv_2(0,s_0)\|_{L^1(I)} \\
&\le |I|\|\p_2\tv_2(t,\cdot)-w'\|_{L^1(I)}+\|\tv_2(t,\cdot)-\tv_2(0,s_0)\|_{L^1(I)} \\
&\le |I|^\frac{3}{2}\|\dist((R^t\nabla\tv(t,\cdot)\nabla F_{|F^{-1}(rt,r\cdot)})^{sym},\tilde\K)\|_{L^2(I)} \\
& \qquad +C'r|I|^2+\|\tv_2(t,\cdot)- \tv_2(0,s_0)\|_{L^1(I)},
\end{split}
\end{equation}
for some constant $C'>0$ depending on $c_1, d_1$.
Gathering \eqref{eq:ineq-gen1} and \eqref{eq:ineq-gen2}, for $r$ sufficiently small, we deduce that 
$$
\frac{c_2}{8}|I|^2 \le \Big(\frac{c_2}{4}-C'r\Big)|I|^2 \le 4|I|^2\tilde F_\epsilon(\tv;\{t\}\times(0,1))^\frac{1}{2}+c_3t|I|+2\sqrt{2}c_4(|I|+t)^2+16t^\frac{1}{2}|I|\tilde F_\epsilon(\tv)^\frac{1}{2}
$$
which, recalling that $2\sqrt{2}c_4(|I|+t)^2=\frac{c_2}{16}|I|^2$, can be further reduced to
$$
\frac{c_2}{16}|I|^2 \le 4|I|^2\tilde F_\epsilon(\tv;\{t\}\times(0,1))^\frac{1}{2}+c_3t|I|+16t^\frac{1}{2}|I|\tilde F_\epsilon(\tv)^\frac{1}{2}.
$$
This recovers an estimate  similar to  \eqref{eq:ref-last-ineq} for $F_\epsilon(\tv)$ and thus yields the claimed lower bound for $F_{\epsilon}(\tilde{v}; \{t\}\times (0,1))$  by a similar argument.

\emph{Step 3: conclusion.} We conclude the estimate $\tilde{F}_{\epsilon}(\tilde v) \geq c \min\{ \epsilon (|\log \epsilon| + 1),1\}$ from \eqref{eq: est claim rotated} using an argument that is similar to Lemma \ref{lem: lb log} but uses the conclusion of the claim in step 2 which differs slightly from the assumptions in Lemma \ref{lem: lb log}.

Let $L \in \N$ be the largest $L$ such that $2^{-L-1} \geq \frac{1}{4c_1} \min\{1, \epsilon (|\log \epsilon| + 1)\}$. Additionally, let $N \in \N$ be the smallest $N$ such that $2^{-N} \leq \frac1{2c_1}$. 
It follows for some $\tilde{c}>0$ that 
\begin{align}
\begin{split}
(L-N - 1) \log(2) &\geq \log(2) - \log \left( \min\left\{ \epsilon (|\log \epsilon| + 1),1 \right\} \right) \\
&\geq \tilde{c} \max\{ 1, 1 - \log(\epsilon) \}. \label{eq: estimate number}
    \end{split}
\end{align}
Next, we estimate 
\begin{equation}\label{eq: energy partition}
	\begin{split}
    2 \tilde F_\epsilon(\tv) &\geq  \sum_{k=N}^{L-1} \int_0^1\int_{2^{-k-2}}^{2^{-k}} \dist^2((R^t\nabla\tilde v(x,y)\nabla F_{|F^{-1}(rx,ry)})^{sym},\tilde\K)dx dy \\
    & \qquad +\epsilon |D^2\tilde v|((2^{-k-2},2^{-k}) \times (0,1)).
    \end{split}
\end{equation}
Fix $N \leq k\leq L-1$. We now distinguish two cases.

\emph{Case 1:}
Let us assume first that we have for all $t \in (2^{-k-1},2^{-k})$ that $\tilde{F}_{\epsilon}(\tv;\{t\} \times (0,1)) \geq c \min\left\{ \frac{\epsilon}{t},1\right\}$. 
By the definition of $N$, if $\epsilon \leq 1$, it holds $\min\left\{ \frac{\epsilon}{t},1\right\} \geq \frac1{4c_1} \frac{\epsilon}{t}$. 
Consequently, in this case we observe
\begin{align*}
&\int_0^1\int_{2^{-k-2}}^{2^{-k}} \dist^2((R^t\nabla\tilde v(x,y)\nabla F_{|F^{-1}(rx,ry)})^{sym},\tilde\K)dx dy+\epsilon |D^2\tilde v|((2^{-k-2},2^{-k}) \times (0,1)) \\
\geq &\int_0^1\int_{2^{-k-1}}^{2^{-k}} \dist^2((R^t\nabla\tilde v(x,y)\nabla F_{|F^{-1}(rx,ry)})^{sym},\tilde\K)dx dy+\epsilon |D^2\tilde v|((2^{-k-1},2^{-k}) \times (0,1)) \\
\geq &\int_{2^{-k-1}}^{2^{-k}} \frac{c}{4c_1} \frac{\epsilon}{t} \, dt = \epsilon \frac{c}{4c_1} \log(2).
\end{align*}
Similarly, if $\epsilon \geq 1$, it holds $\min\left\{ \frac{\epsilon}{t},1\right\} = 1$ for all $t \in (2^{-k-1},2^{-k})$ and hence
\begin{multline*}
\int_0^1\int_{2^{-k-2}}^{2^{-k}} \dist^2((R^t\nabla\tilde v(x,y)\nabla F_{|F^{-1}(rx,ry)})^{sym},\tilde\K)dx dy+\epsilon |D^2\tilde v|((2^{-k-2},2^{-k}) \times (0,1)) \\
\geq c2^{-k-1}\geq  \frac{c}{4 c_1} \min\{1, \epsilon (|\log \epsilon| + 1)\},
\end{multline*}
where we used the definition of $L$.

\emph{Case 2:} Let us now assume that there exists $t \in (2^{-k-1},2^{-k})$ such that $|D^2 \tv|((t/2,t) \times (0,1)) \geq c$. 
It follows that
\begin{multline*}
\int_0^1\int_{2^{-k-2}}^{2^{-k}} \dist^2((R^t\nabla\tilde v(x,y)\nabla F_{|F^{-1}(rx,ry)})^{sym},\tilde\K)dx dy+\epsilon |D^2\tilde v|((2^{-k-2},2^{-k}) \times (0,1)) \\
\geq \epsilon c \geq \epsilon \frac{c}{4c_1} \log(2).
\end{multline*}

Combining the bounds from above, by the conclusion from the claim in step 2, \eqref{eq: estimate number} and \eqref{eq: energy partition}, we find
\begin{align*}
2 \tilde F_\epsilon(\tv) 
&\geq \epsilon \frac{c}{4c_1} \log(2) (L-N-1) \geq \frac{c \tilde{c}}{4c_1} \min\{1,\epsilon\} \max\{ 1, 1-\log(\epsilon) \} \\
&\geq  \frac{c \tilde{c}}{4c_1} \min\{1, \epsilon (|\log \epsilon| + 1)\},
\end{align*}
which concludes the proof of the theorem.
\end{proof}

\section{Proofs of Theorems \ref{thm:main_gen} and \ref{thm:main}}
\label{sec:theorems_proof}
In this section, we present the proofs of Theorems \ref{thm:main_gen} and \ref{thm:main}. We begin by discussing the geometrically linear case and then turn to the geometrically nonlinear one in the next subsection.

\subsection{Proof of Theorem \ref{thm:main_gen}}

Heading towards the proof of Theorem \ref{thm:main_gen}, we first provide the short argument for Lemma \ref{lem:normals}. Essentially, this can be found in \cite{CDPRZZ20}, we repeat it in our setting for the convenience of the reader.

\begin{proof}[Proof of Lemma \ref{lem:normals}]
The presence of an interface between austenite and the martensite variant $e^{(j)}$, $j \in \{1,2,3\}$, reads
\begin{align*}
\tau \cdot e^{(j)} \tau = 0 \mbox{ for some } \tau \in S^1.
\end{align*}
This is equivalent to
\begin{align*}
(Q_{\ell}^t \tau) \cdot \begin{pmatrix} 1 & 0 \\ 0 & -1 \end{pmatrix}  (Q_{\ell}^t \tau) = 0 \mbox{ for some } \tau \in S^1,
\end{align*}
and for some $Q_{\ell}$ as in \eqref{eq:rot} in Section \ref{sec:appl_lin}. Now, up to an action of $Q_{\ell}$, it then suffices to solve the equation
\begin{align*}
v \cdot \begin{pmatrix} 1 & 0 \\ 0 & -1 \end{pmatrix} v = 0 
\end{align*}
for $v \in S^1$.
Up to normalization, the solutions to this are given by $e_1 + e_2$ and $e_1 - e_2$. Applying the action of $Q_{\ell}$ then proves the claimed characterization \eqref{eq:normals} for the possible austenite-martensite interface normals.
\end{proof}

With Lemma \ref{lem:normals} in hand, we turn to the proof of Theorem \ref{thm:main_gen}.

\begin{proof}[Proof of Theorem \ref{thm:main_gen}]

\emph{Step 1: Proof of Theorem \ref{thm:main_gen}(b).}
We begin by proving the claim of Theorem \ref{thm:main_gen}(b). For $\epsilon>1$, the bound follows by considering $v= 0$. For $\epsilon \in (0,1)$ the claimed scaling behaviour is a consequence of the stress-free construction given in \cite{CDPRZZ20}. There it is shown that it is possible to construct an exactly stress-free $BV$ regular (in the gradient) deformation in a certain rotated equilateral triangular domain (which, as in our current setting, only uses rotations as symmetries instead of the full point group consisting of rotations and reflections). By applying a suitable rotation of the form $Q_j$, $j\in \{1,2,3\}$, it is then possible to produce such a construction involving an arbitrary normal of the six possible normals with only linear energy scaling (since the elastic energy vanishes and the $BV$ norm is finite), see Figure \ref{fig:star}.

In order to conclude the argument for Theorem \ref{thm:main_gen}(b), it remains to prove the lower bound for the energy. To this end, we invoke the boundary condition together with the Poincar\'e and Hölder inequalities. 

Indeed, by Poincar\'e and Hölder, we obtain that
\begin{align*}
|D^2 v|(\Omega) \geq \tilde{C}_{\Omega} \|\nabla v - \langle \nabla v \rangle_{\Omega} \|_{L^{\frac{d}{d-1}}(\Omega)}
= \tilde{C}_{\Omega} \|\nabla v  \|_{L^{\frac{d}{d-1}}(\Omega)}
\geq C_{\Omega} \|\nabla v\|_{L^1(\Omega)}.
\end{align*}
Here,
\begin{align*}
\langle \nabla v \rangle_{\Omega} = |\Omega|^{-1} \int\limits_{\Omega} \nabla v(x) dx = 0,
\end{align*}
where we invoked the assumption that $v = 0$ on $\partial \Omega$.

Hence,
\begin{align*}
F_{\epsilon}(v)
& = \epsilon |D^2 v|(\Omega) + \int\limits_{\Omega} \dist^2(\nabla v(x), \tilde{\K} + \text{Skew}(2)) dx\\
& \geq \min\{C_{\Omega},1 \} \min\{\epsilon,1\} \left( \int\limits_{\Omega}|\nabla v(x)| + \dist^2(\nabla v(x), \tilde{\K} + \text{Skew}(2)) dx \right)\\
& \geq C' \min\{C_{\Omega},1 \} \min\{\epsilon,1\} |\Omega | \min\{\dist(0, \tilde{\K} + \text{Skew}(2) ), \dist^2(0, \tilde{\K} + \text{Skew}(2) ) \}.
\end{align*}
In the last estimate we used the fact that there exists a constant $C'>0$ such that $A+B^2 \geq C'\min\{A+B, (A+B)^2\}$ for all $A,B \geq 0$ and that by the triangle inequality we have
\begin{align*}
|\nabla v(x)| + \dist(\nabla v(x), \tilde{\K} + \text{Skew}(2)) \geq \dist(0, \tilde{\K} + \text{Skew}(2) ).
\end{align*}
This concludes the lower bound.\\

\emph{Step 2: Proof of Theorem \ref{thm:main_gen}(a).}
In order to prove the claim of Theorem \ref{thm:main_gen}(a) we split the argument into two steps. \\

\emph{Step 2a: The lower bound.} Firstly, we discuss the lower bound estimate.
To this end, it suffices to verify that condition \eqref{eq:dist-lin} is satisfied which then allows to invoke Theorem \ref{prop:gen_domain_lin}.
In case that $\Omega$ is polygonal, this immediately follows from the fact that $\tilde{\K}_3$ is discrete with only a discrete set of normals satisfying the linearized Hadamard jump conditions while the boundary of $\Omega$ has a normal field which onto contains an open subset of $S^1$. 
In case that $\Omega$ is polygonal with a planar part of the boundary $\Gamma$ such that the associated normal is not in the set $\tilde{\mathcal{N}}$, we also infer that the condition \eqref{eq:dist-lin} is satisfied, as 
\begin{align*}
\tau \cdot e^{(j)} \tau = 0 \mbox{ for some } j\in \{1,\dots,n\},
\end{align*} 
iff $\tau \in \tilde{\mathcal{N}}$.\\

\emph{Step 2b: The upper bound.}
Secondly, we turn to the upper bound in Theorem \ref{thm:main_gen}(a). To this end, we need to provide an explicit microstructure. As in \cite[Section 6]{RZZ16} (see also \cite{G23} and \cite[Section 7]{B3} for related constructions) this can be achieved by a dyadic, greedy covering argument and the building block microstructure from the proof of Theorem \ref{thm:main_gen}(b). More precisely, we greedily cover $\Omega$ by dyadic lattices consisting of the building block microstructures from Step 1. By the domain regularity, it suffices to consider domain patches $\Omega_f$ which are Lipschitz graphs with Lipschitz constant $Lip(f)$ over the base line of the building block construction. Moreover, we may assume that $0 \in \inte(\Omega_f)$ and that $\Omega_f$ is a star-shaped domain with zero being one possible star point. 
Then, we cover $\Omega_f$ by dyadic lattice triangles. Indeed, we set $\Omega_{f,\ell}:= \bigcup\limits_{j=1}^{K_{\ell}} \Delta_{\ell}^j$, where $\Delta_{\ell}^j$ are lattice triangles which are oriented as one of the the admissible building block triangles and are of scale $2^{-\ell}$. By the Lipschitz regularity of the boundary of $\Omega_f$, for $\ell>0$ sufficiently large, we may assume that $\tilde{\Omega}_{\ell}:= (1-C_f 2^{-\ell}) \Omega_f$ is completely covered by such triangles of scale $2^{-\ell}$. Here 
\begin{align*}
(1-C_f 2^{-\ell}) \Omega_f:=\left\{ x \in \Omega: \ (1-C_f 2^{-\ell})^{-1} x \in \Omega \right\}
\end{align*}
and $C_f:= 10 Lip(f)+10$. 
\begin{figure}[t]
\includegraphics[width = 4 cm]{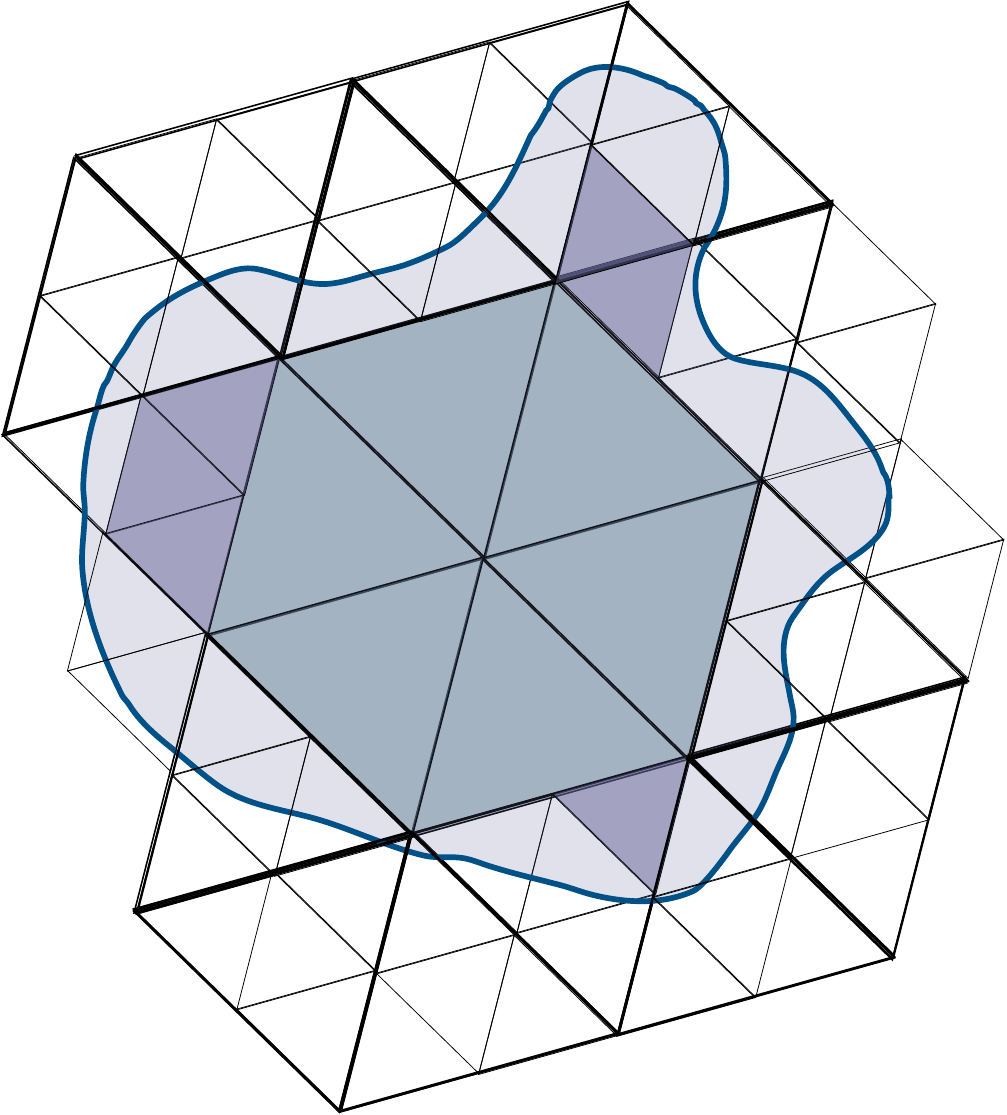}
\caption{A greedy covering of a domain by the triangle building blocks. The dyadic refinement leads to the logarithmic losses in the scaling law.}
\label{fig:covering}
\end{figure}
Furthermore, we observe that by a volume estimate the set $\tilde{\Omega}_{\ell+1}\setminus \tilde{\Omega}_{\ell+1}$ contains at most $(20C_f+ 20)2^{\ell+1}$ many lattice triangles of scale $2^{-(\ell +1)}$.
Thus, we conclude that for any $m\in \N$ it is possible to cover $\Omega$ by such construction up to a boundary layer of size $\sim 2^{-m}$ by using at most $(20C_f+ 20) 2^{\ell}$ triangles of scale $2^{-\ell}$ with $\ell \in \{0,\dots,m\}$. 
In each of these triangles there is no elastic energy and the surface energy is proportional to the side length of the triangle. In the boundary layer we let the deformation be equal to the identity; hence the boundary layer in turn carries an elastic energy proportional to its volume. As a consequence, denoting the resulting deformation at refinement scale $m$ to be $v_m$, we obtain
\begin{align*}
\E(v_m,\Omega) \sim 2^{-m} + \epsilon m.
\end{align*}
Hence, choosing $m \sim -\log(\epsilon)$ implies the claimed upper bound.
\end{proof}

\subsection{Proof of Theorem \ref{thm:main}}

The argument for Theorem \ref{thm:main} follows similarly as for the geometrically linearized case. 

\begin{proof}[Proof of Theorem \ref{thm:main}]
The first claim follows from discreteness and the explicit characterization of the austenite-martensite interfaces given in the sets $\mathcal{N}_3, \mathcal{N}_4$. The second statement and the upper bound construction follows (up to a suitable rotation) from the explicit constructions in \cite{CKZ17} and \cite{CDPRZZ20}.
\end{proof}

\section*{Acknowledgement} 
A.R.~and A.T.~gratefully acknowledge funding by the Deutsche Forschungsgemeinschaft (DFG, German Research Foundation) through SPP 2256, project ID 441068247. All authors  were partially supported by the Hausdorff Institute for Mathematics at the University of Bonn which is funded by the Deutsche Forschungsgemeinschaft (DFG, German Research Foundation) under Germany's Excellence Strategy – EXC-2047/1 – 390685813, as part of the Trimester Program on Mathematics for Complex Materials. A.R. is supported by the Hausdorff Center for Mathematics
which is funded by the Deutsche Forschungsgemeinschaft (DFG, German Research Foundation)
under Germany’s Excellence Strategy – EXC-2047/1.
J.G.~and B.Z.~gratefully acknowledge the support of the Deutsche Forschungsgemeinschaft (DFG, German Research Foundation) via project 195170736 - TRR 109 and BZ via project 211504053 - SFB 1060.

\bibliographystyle{alpha}
\bibliography{citations5}

\end{document}